\documentclass[absolute]{mymathart}
\usepackage{mymathmacros}
\usepackage{color}
\usepackage{graphicx,subfig,commath}
\usepackage[shortlabels]{enumitem}
\usepackage{ifthen}

\newboolean{escher}
\setboolean{escher}{false}

\newboolean{quality}
\setboolean{quality}{true}

\newcommand{\lrg}[1]{#1}

 \newtheoremstyle{mytheoremstyle}
   {9pt}
   {9pt}
   {\it}
   {}
   {\bfseries}
   {.}
   {\newline}
   {}

 \newtheoremstyle{numberedstyle}
   {9pt}
   {9pt}
   {\normalfont}
   {}
   {\bfseries}
   {.}
   {\newline}
   {}

\newcommand{\modd}[1]{|{\rm d}#1|}
\newcommand{\Deriv}{{\rm D}}

\numberwithin{equation}{section}

\newcommand{\Orbit}{\mathcal{O}}

\newcommand{\Arg}{\operatorname{Arg}}

\theoremstyle{mytheoremstyle}
\newtheorem{thm}{Theorem}[section]%
\newtheorem{lem}[thm]{Lemma}%
\newtheorem{cor}[thm]{Corollary}%
\newtheorem{prop}[thm]{Proposition}%
\newtheorem{obs}[thm]{Observation}%

\theoremstyle{numberedstyle}
\newtheorem{defn}[thm]{Definition}%

\newtheorem{exercise}[thm]{Exercise}

\renewcommand{\H}{\mathbb{H}}
\newcommand{\HH}{\mathbb{H}}

\newcommand{\real}{\operatorname{real}}

\usepackage{hyperref}

\title[The exponential map is chaotic]{The exponential map is chaotic:\\
   An invitation to transcendental dynamics}
\author{Zhaiming Shen}
\author{Lasse Rempe-Gillen}
\thanks{This work was supported by an LMS-Nuffield Undergraduate Research
  Bursary. The second author is supported by
  a Philip Leverhulme Prize.}
\subjclass[2000]{Primary 37F10; Secondary 30D05, 37D45}

\begin{document}

\begin{abstract}
  We present an elementary and conceptual proof that the complex exponential map is \emph{chaotic} when considered as a dynamical system 
   on the complex plane. (This \lrg{was conjectured} by Fatou in 1926 and first proved by Misiurewicz 55 years later.) The only background required
   is a first undergraduate course in complex analysis. 
\end{abstract}

 \maketitle

 \section{Introduction}
  Let $x_0$ be any real number, and consider what happens when we repeatedly
   apply the function $f(x) = e^x$:
   \[
       x_0 \mapsto e^{x_0} \mapsto e^{e^{x_0}}\mapsto e^{e^{e^{x_0}}} \mapsto \dots.
   \]
  Clearly this sequence\lrg{, the \emph{orbit} of $z_0$ under $f$,} tends rapidly to infinity. Indeed, if
   $x_n=e^{x_{n-1}}$ denotes the $n$-th term in the sequence, then
   $x_n>2^{n-2}$ for $n\geq 2$.

  So the process may not appear terribly interesting.
   This changes rather drastically \lrg{upon replacing the \emph{real} number $x_0$ by
   a \emph{complex} value $z_0$} and considering the \lrg{sequence}
   \begin{equation}\label{eqn:iteration}
       z_n \defeq e^{z_{n-1}} \qquad (n\geq 1).
   \end{equation}
 (See Section \ref{sec:background} for a reminder of the \lrg{properties}
  of the complex exponential function.) 
  \lrg{In contrast to the real case, not every complex
  orbit tends to infinity:
  for example, there is a point
  $z_0\approx 0.318 + 1.337i$ such that $f(z_0)=z_0$, and hence the sequence
  defined by \eqref{eqn:iteration} is constant for $z_0$.}
 In fact, things turn out to be extremely complicated:
 \begin{thm}[Orbits of the complex exponential map]\label{thm:orbits}
  Each of the following sets is dense in the complex plane:
   \begin{enumerate}[1.]
     \item\label{item:escaping} the set of starting values $z_0$ whose orbit
      (defined by~\eqref{eqn:iteration}) diverges to $\infty$;
     \item the set of starting values $z_0$ whose orbit
         forms a dense subset of the plane;
     \item the set of \emph{periodic points}; i.e.\ starting values $z_0$
      such that $z_{n+k}=z_n$ for some $k>0$ and all $n\geq 0$.
   \end{enumerate}
 \end{thm}

 So, by performing
  arbitrarily small perturbations of any given starting
  point, we can always obtain an orbit
  that is finite, one that accumulates everywhere
  and one that eventually leaves every bounded set! In particular, the
  eventual behaviour of a point $z_0$ under iteration of the exponential
  map is usually \emph{impossible to predict numerically}: when computing the value
  of $f(z_0)$, there will always be a (tiny) numerical error, and according to
  Theorem \ref{thm:orbits}, this error can change the
  long-term behaviour of orbits drastically.

 This type of phenomenon is often referred to as \emph{chaos}.
  It is a typical occurrence in all but the simplest
  ``dynamical systems'' (mathematical systems that change over time according
  to some fixed rule), such as the movement of bodies in the solar system~--
  governed by Newton's laws of gravity~-- or, indeed, seemingly
  simple discrete-time processes such
  as the one we are studying here. There are a number of (inequivalent) definitions
  of ``chaos''; the most widely used, and most appropriate for our purposes, was
  introduced by
  Devaney in 1989 \cite{devaneyintroduction}. This concept, formally
  introduced in Definition \ref{defn:devaneychaos} below, captures precisely the topological
  properties usually associated with chaotic systems. The following result is then
  a consequence of Theorem \ref{thm:orbits}.

 \begin{thm}[The exponential map is chaotic] \label{thm:chaotic}
  The exponential map $f\colon\C\to\C; z\mapsto e^z$ is chaotic in the
   sense of Devaney.
 \end{thm}

 Theorem \ref{thm:orbits} is (a reformulation of) a famous theorem of
  Misiurewicz from 1981 \cite{misiurewiczexp}, which confirmed a conjecture
  stated by Fatou \cite{fatouentire} in 1926. Misiurewicz's proof
  is entirely elementary, but \lrg{not easy:} it relies
  on a sequence of explicit estimates on the exponential map, its
  iterates and their derivatives. An alternative proof was later given
  \lrg{independently in \cite{bakerrippon}, \cite{alexmisha1984kharkov}} (see also \cite{alexmisha}) and \cite{goldbergkeen}. (According to Eremenko, their research
   was directly motivated
    by the desire to give a more conceptual proof of Misiurewicz's theorem.) 
   A third argument \lrg{can be found in} \cite{bergweileretalwandering}. In all \lrg{these newer works}, the result
   arises as part of a more general theorem, and requires a substantial
   amount of background knowledge in complex analysis and complex dynamics.

 The goal of this note is to give a proof of Theorems \ref{thm:orbits} and
  \ref{thm:chaotic} that is both elementary and conceptual. It requires no background
   beyond 
  a first undergraduate
  course in complex analysis, together with some
  facts from
  \emph{hyperbolic geometry} that can be verified in an elementary manner.
  We shall explain the latter carefully in Section \ref{sec:hyperbolic}, after first reviewing the
  action of the exponential map on the complex plane in Section \ref{sec:background}.
 \lrg{Readers already familiar with this background material can dive in straight in with the 
    proofs in Sections~\ref{sec:escaping} to~\ref{sec:periodicpoints}.
  In Section \ref{sec:further}, we briefly mention further results and 
  open questions; since mathematics is learned best by doing, 
  we end with exercises for the reader in Section~\ref{sec:exercises}.}
  We hope that \lrg{our note} will give readers some insights into the
  beautiful phenomena one encounters when studying the dynamics of transcendental functions of one complex variable, and serve
  as an invitation to learn more about this intriguing subject.

\subsection*{Acknowledgements} \lrg{We thank Alexandre Eremenko, Rongbang Huang, Stephen Worsley and the referees for helpful comments.}

\section{Background material: Exponentials, logarithms and chaos}\label{sec:background}

 \subsection*{Basic notation}
  If $f\colon X\to X$ is a self-map of some set $X$, then 
     \[ f^n \defeq \underset{\text{$n$ times}}{\underbrace{f\circ f\circ \dots \circ f}} \]
   \lrg{is called the $n$-th \emph{iterate} of $f$ (for $n\geq 0$).
   In particular, $f^0(x)=x$ for all $x\in X$.} The \emph{orbit} of a point
    $x_0$ is the sequence $(x_0 , f(x_0), f^2(x_0), \dots)$. In the case where $f(z)=e^z$ is the complex exponential map, this coincides precisely
    with the definition in~\eqref{eqn:iteration}. A point $x\in X$ is \emph{periodic} if there is some $n\geq 1$ such that $f^n(x)=x$. 

 \lrg{We use standard notation for complex numbers $z\in\C$. In particular,} $\Arg(z)\in (-\pi,\pi]$ 
   denotes (the principal branch of) the \emph{argument} of $z$, i.e.\
   the angle that the line segment connecting $0$ and $z$ forms with the real axis. Note that $\Arg(z)$ is undefined at $z=0$ and continuous only for
   $z\notin (-\infty,0)$.
   \lrg{We write $\C^* \defeq \C\setminus \{0\}$ for the
  \emph{punctured plane}, and $D_{\delta}(z_0)$ for the round disc of radius $\delta$ around a point $z_0\in\C$; the \emph{unit disc} is
  $\D \defeq D_1(0)$.}

 \subsection*{The complex exponential function}
  \begin{figure}
    {\includegraphics[width=.5\textwidth]{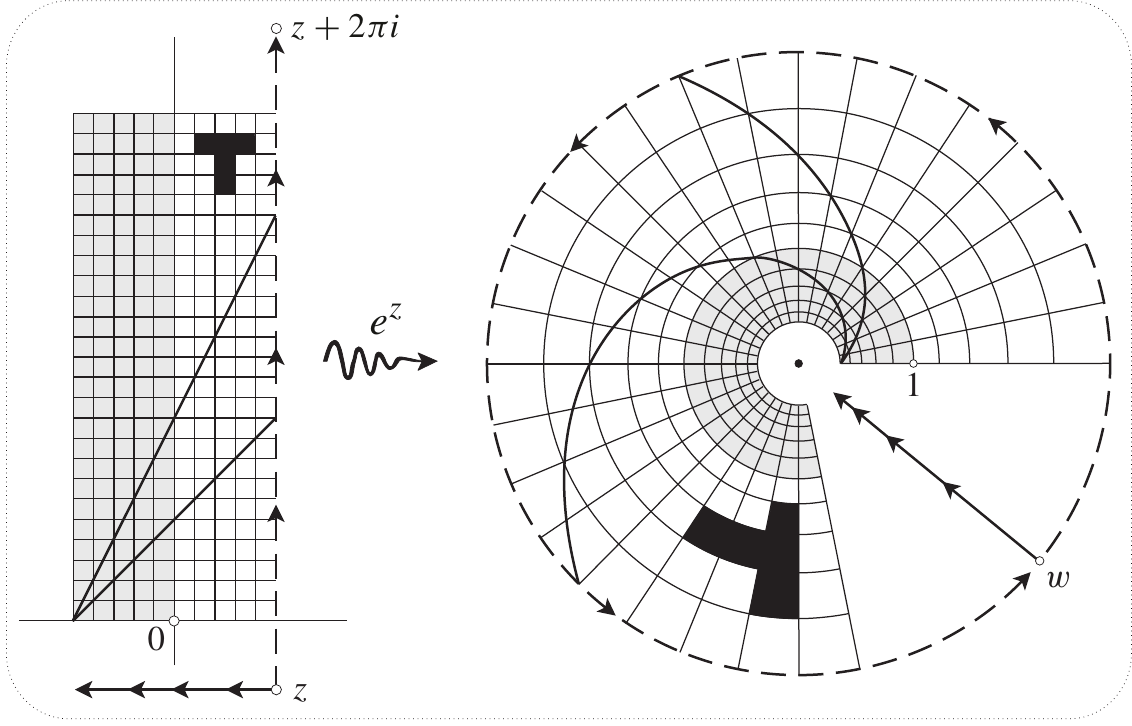}}\hfill
    \caption{\label{fig:exponential}The action of the complex exponential map, as illustrated in Figure [10] on p.\ 82 of 
     \cite{visualcomplexanalysis}. The pattern on the right is the image of the
     checkerboard pattern on the left. (Image courtesy of Tristan Needham.)}
 \end{figure}

  Recall that the complex exponential function is 
    \lrg{the holomorphic function $\exp\colon \C\to\C$ given by}
   \begin{equation} \label{eqn:expdefn}
      \exp(z) \defeq e^{z} = e^x\cdot e^{iy} = e^x \cdot (\cos(y) +i\sin(y)), \qquad\text{where } z=x+iy. \end{equation}
 The representation~\eqref{eqn:expdefn} provides the following geometric interpretation of the action of the exponential map
  (see Figure \ref{fig:exponential}), which the reader should keep in mind:
\begin{itemize}
 \item The function $\exp$ maps horizontal lines to radial rays starting at the origin, and 
   \lrg{wraps vertical lines infinitely often around} concentric circles centered at the origin. The
    modulus $|e^z|$ is large precisely when $\re z$ is large and positive.
 \item In particular, the \emph{right half-plane} $\HH\defeq\{z\in\C\colon  \re z > 0\}$ is mapped (in an infinite-to one manner)
   to the outside of the closed unit disc, and the left half-plane is mapped to the
   punctured unit disc.
 \item The exponential map is strongly \emph{expanding} when $\re z$ is large and positive, and strongly \emph{contracting}
     when $\re z$ is very negative.
\end{itemize}

  \lrg{Since the exponential map spreads the complex plane $\C$ over the punctured plane $\C^*$ in
   an infinite-to-one manner, it is not injective, and hence does not have a  
  well-defined global inverse.
  Instead, we use \emph{branches} of the logarithm: inverse functions to
  injective restrictions of $\exp$  (see \cite[Section~2.VII]{visualcomplexanalysis}). It follows 
  from~\eqref{eqn:expdefn} that such a branch exists wherever 
  there is a continuous choice of the argument. For example, 
   let $\theta\in\R$. Then
     \[ \exp\colon  \{ \zeta \in \C\colon \theta -\pi < \im \zeta < \theta + \pi \} \to 
                       \{\omega \in\C\colon \Arg \omega \not\equiv -\theta\ (\mod 2\pi) \} =: U \]
    is bijective, and hence has a holomorphic inverse $L$ on $U$, given by    \[ L(\omega) \defeq  \log|\omega| + i\cdot\left(\theta + \Arg\frac{\omega}{e^{i\theta}}\right)
         \qquad (\omega\in U).  \] 
   (Here, and throughout, $\log\colon (0,\infty)\to \R$ denotes the natural logarithm.)
    In particular,
     let $\Delta\subset\mathbb{C}$ be a disc that does not contain the origin, and let  
    $\zeta_0\in\mathbb{C}$ with $e^{\zeta_0}\in \Delta$. Taking $\theta = \im \zeta_0$, we see that there is a holomorphic map 
     $L\colon \Delta\rightarrow\mathbb{C}$ with $L(e^{\zeta_0})=\zeta_0$ such that 
     $e^{L(\omega)}=\omega$  and
     $\im \zeta_0 - \pi < \im L(\omega) < \im \zeta_0 + \pi$ 
     for all $\omega \in D$.}

 \lrg{Note that we could replace $\Delta$ by any convex open set
   that omits the origin.
   More generally, branches of the logarithm exist on
   any \emph{simply-connected} domain that does not contain the origin, but
   for us the above cases of discs and slit planes will be sufficient.}

 \subsection*{Devaney's topological definition of chaos}

  We now introduce Devaney's definition of chaos \cite[\S1.8]{devaneyintroduction} that was mentioned in the introduction. This is
  usually stated in the context of \emph{metric spaces} (see below), but we shall restrict to the case of dynamical systems defined
   on subsets of the complex plane. Recall that, if $A\subset X\subset\C$, then $A$ is called 
   \emph{dense} in $X$ if every open set $U\subset\C$ that intersects $X$ also contains a point of $A$.

 \begin{defn}[Devaney Chaos] \label{defn:devaneychaos}
   Let $X\subset\C$ be infinite, and let 
   $f\colon X\to X$ be continuous. We say that $f$ is \emph{chaotic} (in the
   sense of Devaney) if the following two conditions are satisfied.
   \begin{enumerate}[(a)]
     \item The set of periodic points of $f$ is \emph{dense} in $X$.\label{item:chaos_periodic}
     \item The function $f$ is \emph{topologically transitive}; that is,
       for all open sets $U,V\subset \C$ that intersect $X$, there is a point $z\in U\cap X$ and $n\geq 0$ such that
       $f^n(z)\in V$.\label{item:chaos_trans}
   \end{enumerate}
 \end{defn}
   Topological transitivity means precisely that we can move from any part of the
   space $X$ to any other by applying the function $f$ sufficiently often. This property clearly follows from the existence
   of a dense orbit (see Exercise \ref{ex:denseorbit}). In particular, Theorem \ref{thm:chaotic} is a consequence of 
   Theorem \ref{thm:orbits}. (However, in Section \ref{sec:transitivity}, we in fact argue in the converse direction~-- we first prove
   topological transitivity of the exponential map directly, and then deduce the existence of dense orbits.)

\subsection*{Sensitive dependence and spherical distances}

  Devaney's original definition of chaos included a third condition: \emph{sensitive dependence on initial conditions}. 
   It is shown in
   \cite{banksetal} that
   this property is a consequence of topological transitivity and density of periodic points, which is why we were able to omit 
   it in Definition~\ref{defn:devaneychaos}.

  It is nonetheless worthwhile to discuss sensitive dependence on initial conditions, since it
    encapsulates precisely the idea of ``chaos'' discussed in the introduction: Two points that
   are close together may end up a definitive distance apart after sufficiently many applications of the function $f$. (This phenomenon
   has become known in popular culture as the ``butterfly effect''.) Moreover, in the case of the exponential map, we shall be able to 
   establish sensitive dependence before proving either of the remaining two conditions. 

  To give the formal definition, we shall use the notion of a \emph{distance function} (or \emph{metric}) $d$ on a set $X$, as introduced 
    e.g.\ in \cite[Chapter 2]{burkillburkill} or \cite[Chapter 2]{babyrudin}. Readers unfamiliar with this definition need not despair:
    On the one hand, it is only used in this subsection; on the other, it will be sufficient to think of such a function informally
    as a formula that defines some notion of \emph{distance} between two points of $X$. The simplest example of a distance function on $\C$ is
    the usual one, namely \emph{Euclidean distance} $d(z,w) := |z-w|$. 

\begin{defn}[Sensitive dependence on initial conditions]\label{defn:sensitivedependence}
   Let $X\subset\C$, let $f:X\to X$ be continuous, and let $d$ be a distance function on $X$. 
   We say that $f$ exhibits \emph{sensitive dependence on initial conditions} (with respect to $d$) if there exists a constant
     $\delta>0$ with the following property: For every non-empty open set $U\subset X$, there are points
      $x,y\in U$ and some $n\geq 0$ such that $d(f^n(x),f^n(y))\geq \delta$.
\end{defn}

 If we intend to use this definition,  we ought to clarify which distance function we intend to use on the complex plane, and it is 
    tempting to choose Euclidean distance. 

This turns out to be a poor choice: even
    the ``uninteresting'' \emph{real} exponential map discussed at the very beginning of this paper has sensitive dependence with respect to this distance
   (Exercise \ref{ex:realexp}).
  The trouble is that, when orbits tend uniformly to infinity
    for a whole neighbourhood of our starting value, we would consider the corresponding
   behaviour to be \lrg{``stable''}, but orbits \lrg{may end up} extremely far apart 
    in terms of Euclidean distance.
     This issue is resolved by instead using the so-called \emph{spherical distance}, which is 
     obtained by adjoining a single point
    at $\infty$ to the complex plane~-- so a point $z$ is close to $\infty$ whenever $|z|$ is large~--
    and thinking of the resulting space
    $\C\cup\{\infty\}$ as forming a sphere in $3$-space. For this reason, 
     the space $\C\cup\{\infty\}$ is \lrg{called} the \emph{Riemann sphere}.

 \lrg{Instead of introducing spherical distance formally, let us use the 
   informal picture above to decide what sensitive dependence with respect to the sphere
    should mean. If $\zeta, \omega\in\C$ are close to each other on the sphere,
    then there are only two possibilities: either both points 
    are close to infinity, or they are also close in the Euclidean sense. Using this observation,
    \lrg{we define} sensitive dependence on the sphere
    directly and axiomatically:}

 \begin{defn}[Sensitive dependence with respect to spherical distance]\label{defn:sphericalsensitive}
   Let $f\colon \C\to\C$ be a continuous function. We say that $f$ has \emph{sensitive dependence with respect to spherical distance}
    if there exist $\delta>0$ and $R>0$ with the following property. For every non-empty open set $U\subset\C$, there are $z,w\in U$ and
     $n\geq 0$ such that $|f^n(z)|\leq R$ and $|f^n(z)-f^n(w)|\geq \delta$.
 \end{defn}

 \lrg{It turns out (see Exercise \ref{ex:sphericalsensitiveimpliesallothers}) that Definition \ref{defn:sphericalsensitive}
  implies sensitive dependence in the sense of Definition~\ref{defn:sensitivedependence}~--
  no matter which distance function we use on the complex plane.}

 \section{A brief introduction to hyperbolic geometry}
 \label{sec:hyperbolic}

 \emph{Hyperbolic geometry} is a beautiful and powerful tool in one-dimensional complex analysis (as well as in higher-dimensional geometry). 
  When discussing spherical distance in the previous section,
  we briefly
  encountered the idea of using a different notion of distance to the ``standard'' one;
  hyperbolic geometry is another example of this. If $U$ is any open subset of the complex plane that omits more than one point, then
  there is a natural notion of distance on $U$, called the \emph{hyperbolic metric}. (For those who know differential geometry, this 
  is the unique complete conformal metric of constant curvature $-1$ on $U$.) We only need a few elementary facts, all of which can be
  proved using elementary complex analysis. We shall first motivate these statements
  and then collect \lrg{them} in Theorem \ref{thm:pick} and Proposition \ref{prop:hypexamples}.
  For a more detailed introduction to the hyperbolic metric, we refer
   to the book \cite{andersonhyperbolic} or the article \cite{beardonminda}.

  \lrg{Our starting point is the following classical 
     consequence of the standard maximum modulus principle of complex analysis; see
  \cite[Section~3.2]{fishercomplexvariables} or \cite[Section~7.VII]{visualcomplexanalysis}.}
 \begin{lem}[Schwarz lemma] \label{lem:schwarz}
  Suppose that $f\colon \D\rightarrow\D$ is holomorphic and $f(0)=0$. Then either
    \begin{enumerate}
        \item $|f(z)|<|z|$ for every non-zero $z$ in $\D$, and $|f'(0)|<1$, or
        \item there is a real constant $\theta$ such that $f(z)=e^{i\theta}z$ for all $z$, and $|f'(0)|=|e^{i\theta}|=1$.
    \end{enumerate}
 \end{lem}

 This lemma can be very useful, but its generality is limited because of the requirement that $f$ should fix $0$, and because its conclusion concerns only
   the derivative \lrg{at the origin. However,} we can move any point $a\in\D$ to zero using a \emph{M\"obius transformation}
   \begin{equation}\label{eqn:mobius}
       M\colon  \D\to\D; \qquad  z\mapsto e^{i\theta} \frac{z-a}{1-\overline{a}z }
   \end{equation}
  (where $\theta\in\R$ is arbitrary). So, if $f\colon \D\to\D$ is \emph{any} holomorphic function, we can pre- and post-compose $f$ with suitable
  M\"obius transformations and apply the Schwarz lemma.
\lrg{Using the chain rule to determine the derivative of the composition, we see that
  \begin{equation}\label{eqn:schwarzpick} |f'(z)|\cdot \frac{1-|z|^2}{1-|f(z)|^2} \leq 1
  \qquad\text{for all $z\in\D$.} \end{equation}}

  We can interpret~\eqref{eqn:schwarzpick} as saying that the derivative of $f$ is at most $1$ when calculated \emph{with respect to a different
   notion of distance} (in the difference quotient usually used to define $|f'|$). More precisely, we call the expression
    \begin{equation}\label{eqn:hypmetricD} \frac{2\modd{z}}{1-|z|^2} \end{equation} 
   the \emph{hyperbolic metric} on $\D$. The idea is that if we have an ``infinitesimal change'' at the point $z$, then its corresponding
   size in the hyperbolic metric is obtained by multiplying its Euclidean length by the quantity
      \begin{equation}\label{eqn:densityD} \rho_{\D}(z) \defeq \frac{2}{1-|z|^2}, \end{equation} 
    called the \emph{density} of the hyperbolic metric.\footnote{%
     The factor $2$ in~\eqref{eqn:densityD} is simply a normalization that \lrg{ensures} that this metric has \emph{curvature} $-1$, rather than some other
     negative constant. It could just as easily be omitted for our purposes, in which case all subsequent densities will also lose a factor of $2$.}
    This can be made precise using the notions of differential geometry (formally, the metric is a way to measure the length of tangent vectors),
    but we can treat~\eqref{eqn:hypmetricD} simply as a formal expression. 

 \begin{remark}
  \lrg{Although this expression is called a ``metric'', it is not a ``distance function'' in the sense 
    of the preceding section. However, it naturally gives rise to such a distance via the 
    notion of \emph{arc-length}; see  \cite[Chapter~3]{andersonhyperbolic}.
    We note that the \emph{spherical metric} can be
    similarly introduced via a conformal metric; that is, a metric that is a scalar multiple 
    of the Euclidean metric at any point, where the scaling factor may depend on the point.}
\end{remark}

  With our new notation, formula~\eqref{eqn:schwarzpick} states, in beautiful simplicity,
    that a holomorphic function $f\colon \D\to\D$ has hyperbolic derivative
    at most $1$ at every point, with equality if and only if $f$ is a M\"obius transformation. What is even better is that we can transfer the metric
    to other domains.

  \begin{defn}[Simply-connected domains] \label{defn:simplyconnected}
    An open connected set $U\subsetneq \C$ is called \emph{simply-connected} if
     there is a conformal isomorphism (i.e., bijective holomorphic function) $\phi\colon  U \to \D$.
 \end{defn}
 \begin{remark}
   Usually, $U$ is called simply-connected if  $\C\setminus U$ has no bounded components; i.e., $U$ has no holes. The
    \emph{Riemann mapping theorem} \cite[Chapter 6]{ahlforscomplexanalysis} 
    states that the two notions are equivalent. Our definition allows us
     to avoid using the Riemann mapping theorem, which is often not treated in a first course on complex analysis.
 \end{remark}

 We can now state the following result, which collects the key properties of the hyperbolic metric.
   \lrg{Its proof is elementary, using only the Schwarz lemma and the fact that
   the
     conformal automorphisms of $\D$ are precisely the M\"obius transformations   
    from~\eqref{eqn:mobius}; see Exercise \ref{ex:mobius}. We leave it to the reader to fill in the details, or to consult \cite[Theorem~6.4]{beardonminda}.}

 \begin{thm}[Pick's theorem] \label{thm:pick} 
  For every simply-connected domain $U\subset\C$, there exists a unique conformal metric
     $\rho_U(z)\modd{z}$ on $U$, called the \emph{hyperbolic metric}, such that the following hold.
   \begin{enumerate}
      \item $\rho_{\D}(z) = \frac{2}{1-|z|^2}$ for all $z\in\D$; \label{item:diskmetric}
      \item if $f\colon U\to V$ is holomorphic, then $f$ does not increase the hyperbolic metric. \lrg{I.e.,}
           \[ \|\Deriv f(z)\|^V_U \defeq |f'(z)|\cdot \frac{\rho_V(f(z))}{\rho_U(z)} \leq 1; \] \label{item:contraction}
     \item for any $z\in U$ and any $f$ as above, we have
           $\|\Deriv f(z)\|^V_U=1$ if and only if $f$ is a conformal isomorphism between $U$ and $V$;\label{item:isomorphism}
     \item if $U\subsetneq V$, then $\rho_U(z)>\rho_V(z)$ for all $z\in U$.\label{item:inclusion}
   \end{enumerate}
 \end{thm}
\ifthenelse{\boolean{escher}}{%
\ifthenelse{\boolean{quality}}{%
\begin{figure}%
  \subfloat[The unit disc $\D$]{\includegraphics[width=.45\textwidth]{CircleLimit4_Disk}}\hfill
   \subfloat[{The slit plane $\C\setminus(-\infty,0]$}]{\includegraphics[width=.45\textwidth]{CircleLimit4_Slitplane}}\\
   \subfloat[The right half-plane $\HH$]{\includegraphics[width=.45\textwidth]{CircleLimit4_RHalfplane}}\hfill
  \subfloat[The strip $\{z\in\C\colon|\im z|<\pi\}$]{\includegraphics[width=.45\textwidth]{CircleLimit4_Strip_cut}}
\caption{\label{fig:hyperbolicsizes}The hyperbolic metric in the unit disc and the simply-connected domains 
    from Proposition \ref{prop:hypexamples}.
    \lrg{Angels} (and devils) all have the same
    hyperbolic size; their Euclidean sizes shrink proportionally with the distance to the boundary. The images were created by Vladimir Bulatov
    (\texttt{www.bulatov.org}) based on M.~C.~Escher's \emph{Circle Limit IV}.}
\end{figure}}{%
\begin{figure}%
  \subfloat[The unit disc $\D$]{\includegraphics[width=.45\textwidth]{CircleLimit4_Disk_sm}}\hfill
   \subfloat[{The slit plane $\C\setminus(-\infty,0]$}]{\includegraphics[width=.45\textwidth]{CircleLimit4_Slitplane_sm}}\\
   \subfloat[The right half-plane $\HH$]{\includegraphics[width=.45\textwidth]{CircleLimit4_RHalfplane}}\hfill
  \subfloat[The strip $\{z\in\C\colon|\im z|<\pi\}$]{\includegraphics[width=.45\textwidth]{CircleLimit4_Strip_cut_sm}}
\caption{\label{fig:hyperbolicsizes}The hyperbolic metric in the unit disc and the simply-connected domains 
    from Proposition \ref{prop:hypexamples}.
    \lrg{Angels} (and devils) all have the same
    hyperbolic size; their Euclidean sizes shrink proportionally with the distance to the boundary. The images were created by Vladimir Bulatov
    (\texttt{www.bulatov.org}) based on M.~C.~Escher's \emph{Circle Limit IV}.}
\end{figure}}}{%
\begin{figure}
 \subfloat[The unit disc $\D$]{\includegraphics[width=.45\textwidth]{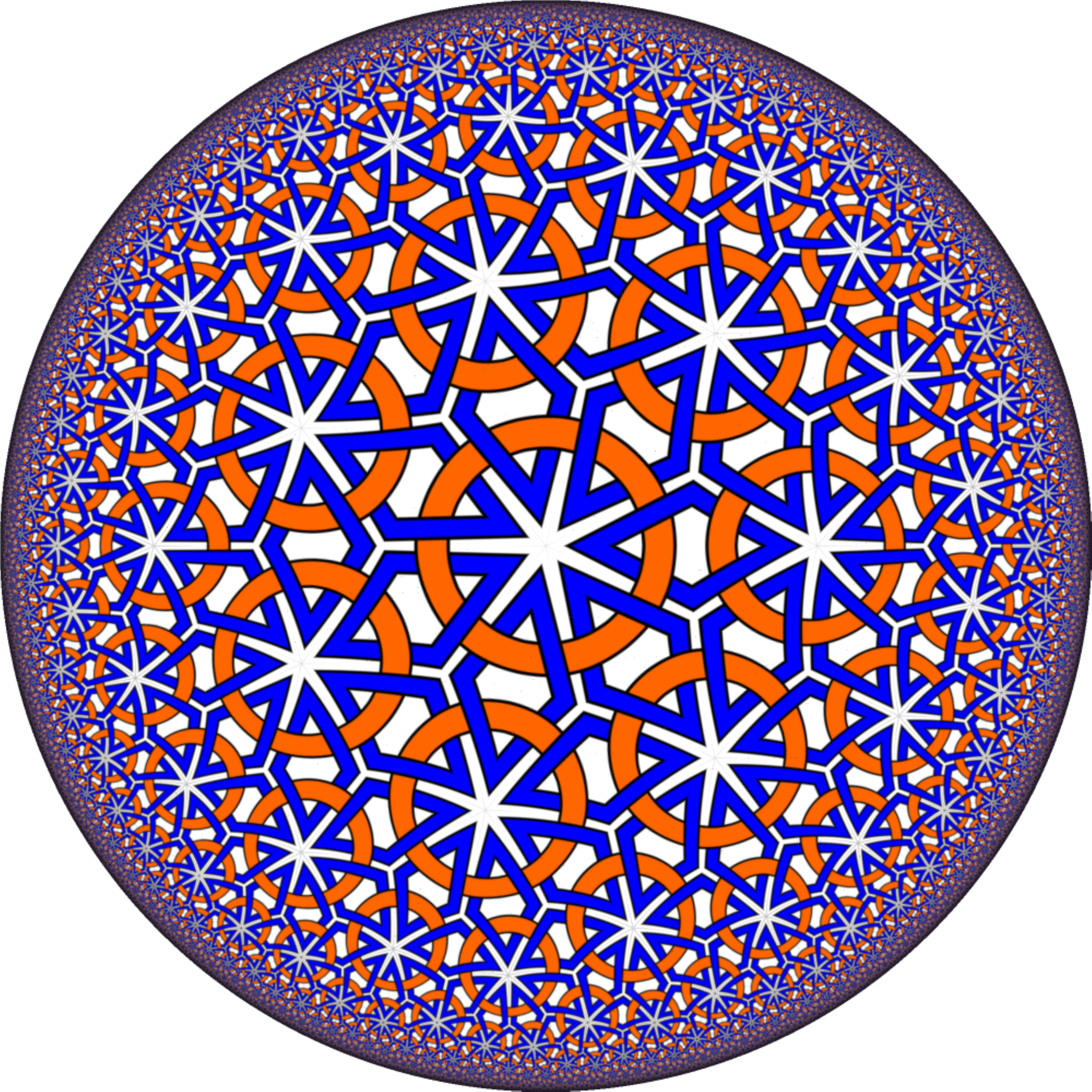}}\hfill
  \subfloat[{The slit plane $\C\setminus(-\infty,0]$}]{\includegraphics[width=.45\textwidth]{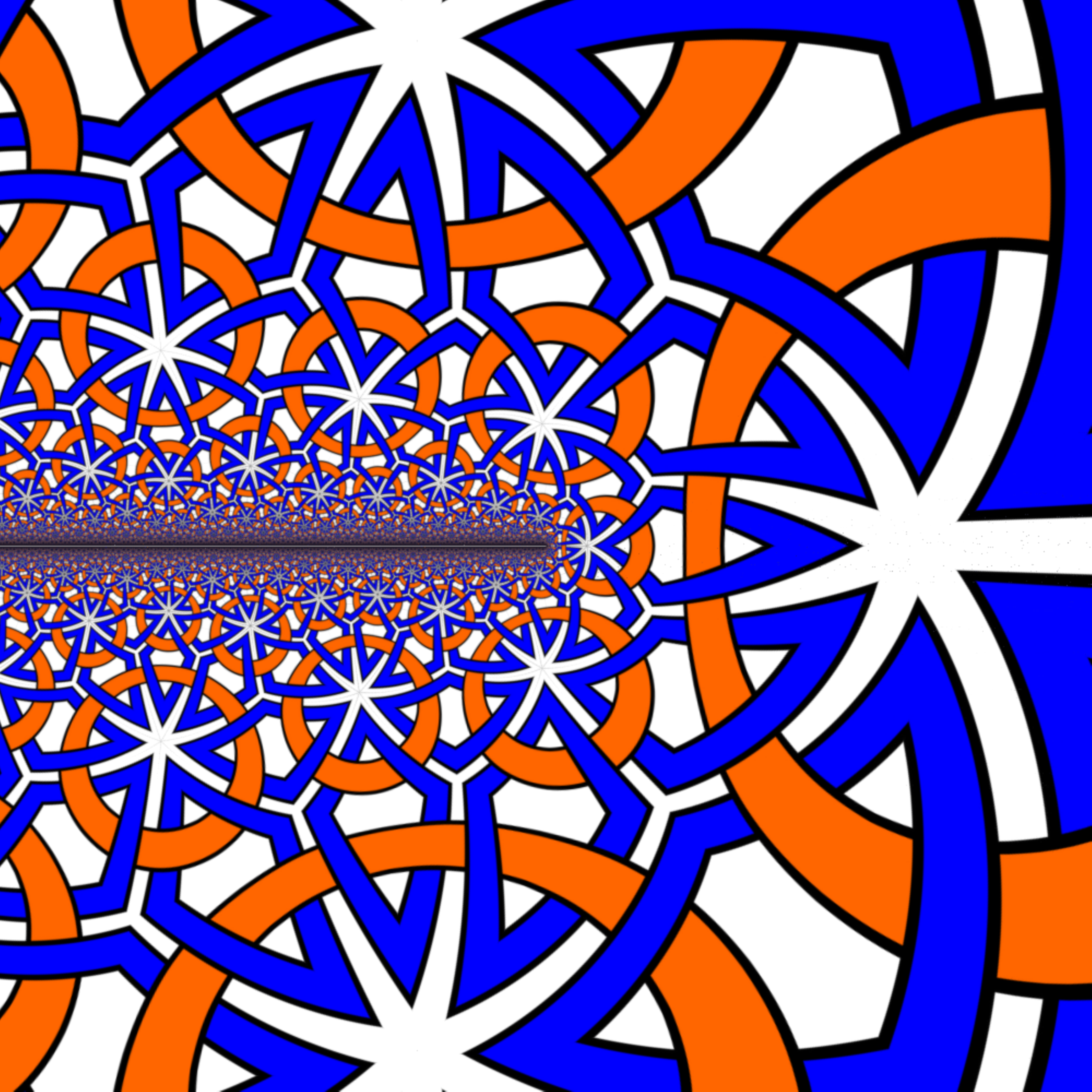}}\\
   \subfloat[{The right half-plane $\HH$}]{\includegraphics[width=.45\textwidth]{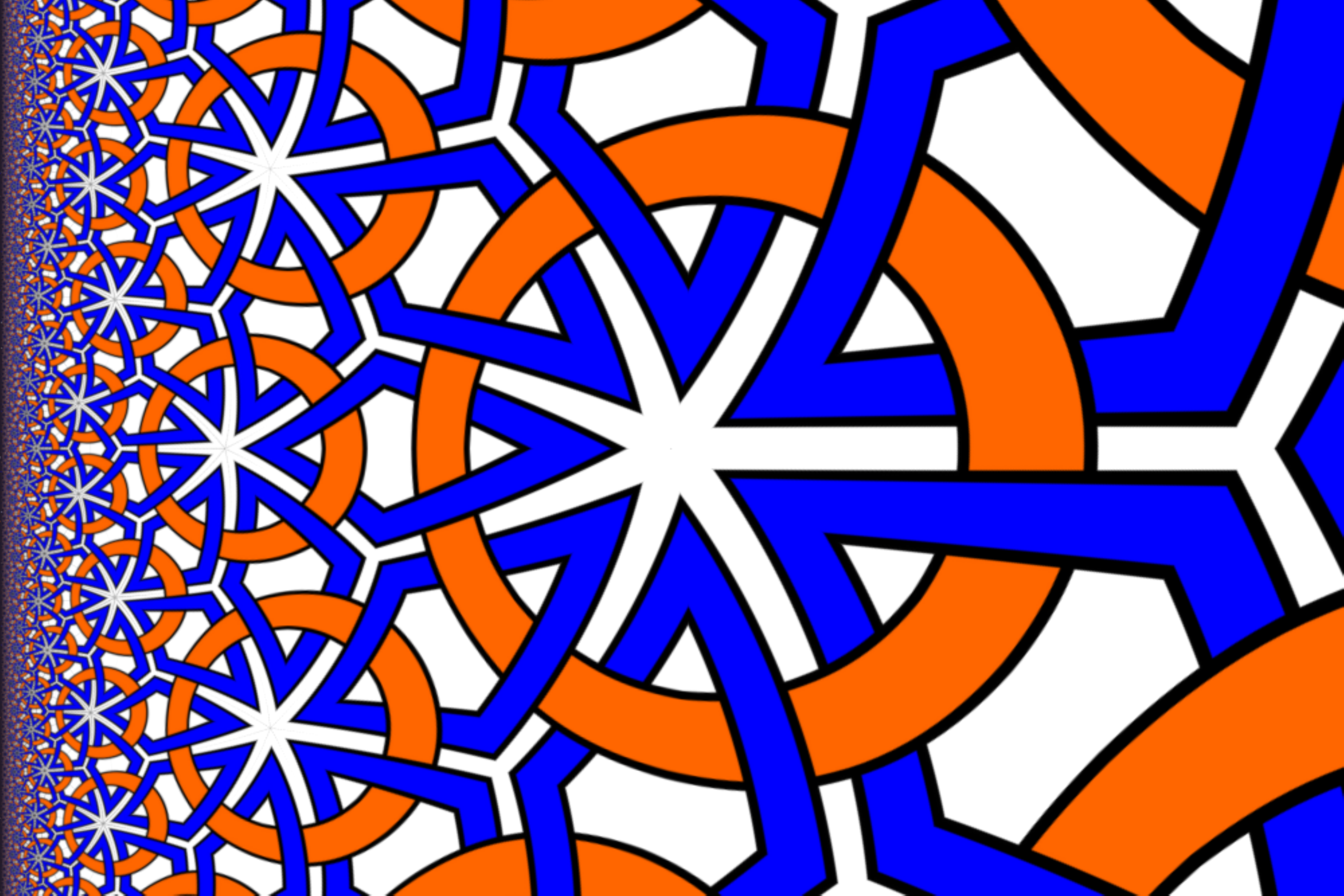}}\hfill
   \subfloat[The strip $\{|\im z|<\pi\}$]{\includegraphics[width=.45\textwidth]{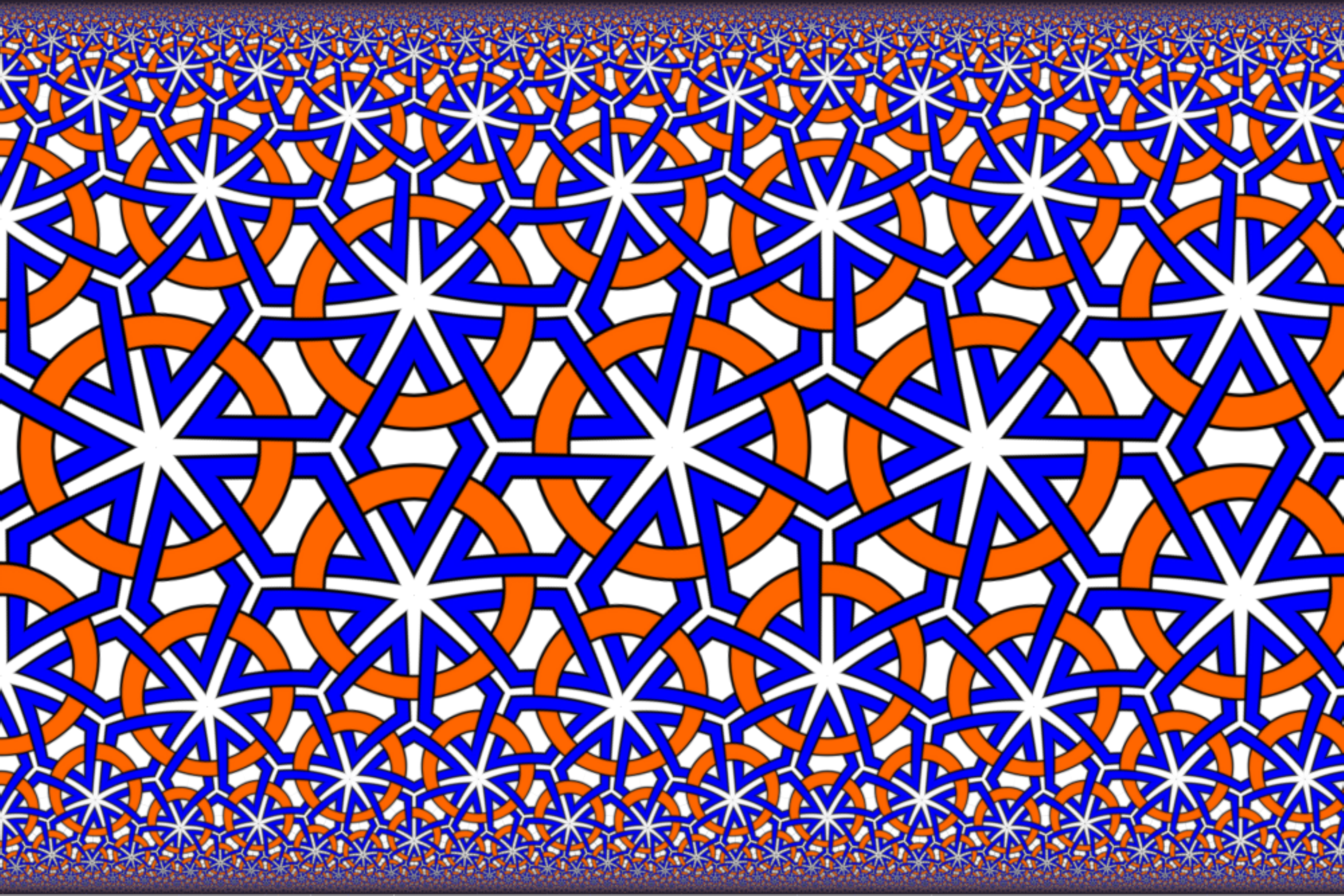}}
\caption{\label{fig:hyperbolicsizes}The hyperbolic metric in the unit disc and the simply-connected domains from Proposition \ref{prop:hypexamples}.
    Red rings all have the same
    hyperbolic size; their Euclidean sizes shrink proportionally with the distance to the boundary. The images were kindly provided by Vladimir Bulatov
    (\texttt{www.bulatov.org}).}
\end{figure}}

 A key property of the hyperbolic metric on a simply-connected domain is that the density $\rho_U(z)$ is inversely proportional to the
   distance of $z$ to the boundary of $U$. In other words, suppose that
   a figure of constant hyperbolic size moves towards the boundary of $U$. Then its
   \emph{Euclidean} size decreases proportionally with the distance to the boundary. (See Figure \ref{fig:hyperbolicsizes}.)  This is a general theorem~--
   see \cite[Formula (8.4)]{beardonminda}~--
   but in the domains that we shall be using, it can be checked explicitly from the following formulae.

 \begin{prop}[Examples of the hyperbolic metric]\mbox{}\label{prop:hypexamples}
  \begin{enumerate}
   \item For the right half plane $\HH=\{z\in\C\colon  \re z >0\}$,
      $\displaystyle{\rho_{\HH}(z) = \frac{1}{\re z}}$.
   \item For a strip of height $2\pi$:
     $\displaystyle{\rho_{\{z\in\C\colon|\im z|<\pi\}}(z) = \frac{1}{2\cos(\im(z)/2)}}$.
   \item For the \emph{positively/negatively slit plane}:
     \[ \rho_{\C\setminus [0,\infty)}(z) = \frac{1}{2|z|\sin(\arg(z)/2)}; \quad
         \rho_{\C\setminus (-\infty,0]}(z) = \frac{1}{2|z|\cos(\Arg(z)/2)}. \]
      (In the first case, the branch of the argument should be taken to range between $0$ and $2\pi$.)
  \end{enumerate}
 \end{prop}
 \begin{proof}
  This can easily be verified by explicit computation using Theorem \ref{thm:pick} (\ref{item:isomorphism}), using the following conformal isomorphisms:
     \begin{align*}
          &\phi_1\colon \H\to \D; \quad z\mapsto \frac{1-z}{1+z}; &\qquad
          &\phi_2\colon S\to \H; \quad z\mapsto e^{\frac{z}{2}}; \\
          &\phi_3\colon \H\to\C\setminus[0,\infty); \quad z\mapsto -z^2; &\qquad
          &\phi_4\colon \C\setminus [0,\infty)\to \C\setminus (-\infty,0]; \quad z\mapsto -z,
     \end{align*}
   where $S=\{z\in\C\colon  |\im z|<\pi\}$.
   As an example, we consider the case of \lrg{$\phi_2$}, and leave the remaining calculations to the reader. We have
      \begin{align*} \rho_S(z) &= \rho_{\H}(\phi_2(z))\cdot |\phi_2'(z)| =
                           \frac{|\phi_2(z)|}{2\re \phi_2(z)} = \frac{1}{2\cos(\arg(\phi_2(z)))} =
                \frac{1}{2\cos(\im(z)/2)}. \qedhere \end{align*}
 \end{proof}

 \section{Escaping points and sensitive dependence} \label{sec:escaping}

 For the remainder of the article
 (unless stated otherwise), $f$ will always denote the exponential map
   \[ f\colon \C\to \C; z\mapsto e^z. \]
  We shall now prove \lrg{the first part} of Theorem
  \ref{thm:orbits}, namely that the escaping set
   \[ I(f)\defeq\{z\in\mathbb{C}\colon f^n(z)\to\infty\} \]
   is dense in the complex plane.
\lrg{The proof
 that we present here
  was briefly outlined in a paper of Mihaljevi\'c-Brandt and the
  second author \cite[Remark on pp.~1583--1584]{wandering}.}

 \begin{thm}[Density of the escaping set]\label{thm:escapingdense}
    The set $I(f)$, consisting of those points $z_0$ whose $f$-orbits converge to infinity, is
     a dense subset of the complex plane $\C$.
 \end{thm}

  The idea of the proof can be outlined as follows. Let $D\subset\C$ be
   any small round disc. Since $\R\subset I(f)$,
   and any preimage of an escaping point is also escaping, there is nothing
   to prove if $U$ contains a point $z$ whose orbit contains a point on
    the real axis.

  Otherwise, $D$, $f(D)$ and all forward images $f^n(D)$ are 
   contained in the slit plane
   \[ U\defeq\mathbb{C}\setminus[0,\infty). \]
   Note that this set $U$ is backward invariant under $f$; i.e.,
    $f^{-1}(U)\subseteq U$. It follows that every branch $L$ of the logarithm
    on $U$ is a holomorphic map $L\colon U\to U$, and hence
    locally contracts the hyperbolic metric of $U$,
    as discussed in the previous section.

  Moreover, if \lrg{$D\cap I(f)=\emptyset$,} then the \lrg{sequence of domains
    $f^n(D)$
    necessarily has at least one finite accumulation point}. Using
    the geometry of $U$, we shall see that this implies that the map
    $f$ expands 
    the hyperbolic metric by a definite factor
    infinitely often along the orbit of $D$. On the other
    hand, the hyperbolic derivative of $f^n$ with respect to $U$ remains
    bounded as $n\to \infty$ by Pick's theorem. This yields the desired
    contradiction.

 From this summary of the proof, it is clear that understanding the expansion of the
   hyperbolic metric of $U$ by $f$ plays a key role in the argument. In the following
   lemma, we investigate where this expansion takes place. 

   \begin{lem}[The exponential map expands the hyperbolic metric] \label{lem:hyper}
    \lrg{The complex exponential map $f$ locally expands the hyperbolic metric on the domain $U\defeq\mathbb{C}\setminus[0,\infty)$. That is, $\|\Deriv f(\zeta)\|_{U}^{U}>1$
    for all $\zeta \in f^{-1}(U)$.} 

      Moreover,
     suppose that 
     $\zeta_n\in f^{-1}(U)$ is a sequence with
     $\|\Deriv f(\zeta_n)\|_{U}^{U}\to 1$ as $n\to\infty$.
     Then $\min\bigl(|\zeta_n|,\Arg(\zeta_n)\bigr)\to 0$ as $n\to\infty$.
   \end{lem}
   \begin{proof}
    \lrg{We shall give two justifications of this result. One is more
     conceptual and uses facts about the hyperbolic metric that were discussed, though 
     not necessarily explicitly proved,
     in the previous section. The second is a completely
     elementary calculation; however, it hides some of the intuition.}

    To give the first argument, let $S$ be a connected component of
     the set
     \[ V\defeq f^{-1}(U)=\{a+ib\colon  \frac{b}{2\pi}\notin\Z \}\subset U. \]
    Then $S$ is a strip of height $2\pi$, and
     $f\colon S\to U$ is a conformal isomorphism.
     Let $\zeta\in S$ and set $\omega \defeq  f(\zeta)\in U$. By Pick's theorem,
     $f|_S$ is a hyperbolic isometry between $S$ and $U$, and hence
    \begin{equation} \label{eqn:expansion} \|\Deriv f(\zeta)\|_{U}^{U} =
    \|\Deriv f(\zeta)\|_S^U \cdot \frac{\rho_S(\zeta)}{\rho_U(\zeta)} =
         \frac{\rho_S(\zeta)}{\rho_U(\zeta)} > 1.    \end{equation}

    Recall that the density at $z$
     of the hyperbolic metric in a simply
     connected domain $G$ is comparable to
     $1/\dist(z,\partial G)$. (For $U$ and $S$, this can be verified
     explicitly, using Proposition \ref{prop:hypexamples}.) We apply this fact twice: 
     Notice first that every point in the strip $S$ has distance at most
     $\pi$ from the boundary, and hence 
     $\rho_S$ is bounded from below by a positive constant.
     So if
     $\zeta_n$ is a sequence as in the claim, then, using the same fact for $\rho_U$, the distance
      $\dist(\zeta_n,[0,\infty)) = \dist(\zeta_n,\partial U) \leq \operatorname{const}/\rho_U(\zeta_n)$
      remains bounded as $n\to\infty$. Also, by~\eqref{eqn:expansion}, $\zeta_n$ does not accumulate at any point of $U$. 
    A sequence of points converging to $\infty$ while remaining within a bounded distance from the positive real axis must 
      have arguments tending to $0$, and 
      a point that is close to a finite point of $\partial U$ either has argument close to $0$ or is close to $0$ itself. So 
      $\min\bigl(|\zeta_n|,\Arg(\zeta_n)\bigr)\to 0$, as claimed.

   On the other hand, the claims can be verified directly from the
    formulae.  Indeed, let us write
    $\zeta=re^{i\theta}$ with $\theta\in(0,2\pi)$; then
     \[ \omega = f(\zeta )=e^{re^{i\theta}}=e^{r(\cos\theta+i\sin\theta)}. \]
   Now we compute the hyperbolic derivative directly, using Proposition \ref{prop:hypexamples}:
    \[
   \|\Deriv f(\zeta)\|_U^U =|f'(\zeta)|\cdot
                \frac{2|\zeta|\cdot\sin(\frac{1}{2}\arg(\zeta))}
              {2|e^{\zeta}|\cdot\sin(\frac{1}{2}\arg(e^{\zeta}))}
                 = \frac{r\cdot\sin(\theta/2)}{\sin(\frac{1}{2}\arg(e^{\zeta}))},
     \]
  where $\arg(\zeta)$ and $\arg(e^{\zeta})$ are chosen in the range $(0,2\pi)$. 
  To further simplify this expression, observe that $\arg(e^{\zeta}) \equiv r\cdot \sin(\theta)\ (\mod 2\pi)$. Since $|\sin|$ is $\pi$-periodic, we hence see that
   $|\sin(\arg(e^{\zeta})/2)|=\sin(r\sin(\theta)/2)$. Furthermore, 
    $|\sin(x)|\leq |x|$ for all $x\in\R$, so
        \begin{align}\label{eqn:eta}
            \|\Deriv f(\zeta)\|_U^U &=
        \frac{r\cdot\sin(\theta/2)}{|\sin(\frac{r}{2}\cdot\sin\theta)|}
            \geq\frac{r\cdot\sin(\theta/2)}{\frac{r}{2}\cdot|\sin\theta|}
\\\notag
             &=\frac{2\cdot\sin(\theta/2)}{|\sin\theta|}=\frac{1}{|\cos(\theta/2)|}>1.
        \end{align}
  (In the final equality, we used the trigonometric formula
    $\sin(2x)=2\sin(x)\cos(x)$.)
    If $\zeta_n$ is such that $\|\Deriv f(\zeta_n)\|_U^U\to 1$, then
   \eqref{eqn:eta} implies \lrg{$|\cos(\theta_n/2)|\to 1$, where $\theta=\arg(\zeta_n)$, and hence
   $\Arg(\zeta_n)\to 0.$}
   \end{proof}
 \begin{remark}
   Observe that the second argument yields the stronger conclusion
      $\Arg(\zeta_n)\to 0$. 
  In fact, a slightly more careful look at the estimates shows
   that $\|\Deriv f(\zeta_n)\|_U^U\to 1$
   if and only if $\Arg(\zeta_n)\to 0$ and $\im \zeta_n\to 0$ 
     (see Exercise \ref{ex:derivativeestimate}). However,
   this will not be required in the proofs that follow.
 \end{remark}

   \begin{proof}[Proof of Theorem \ref{thm:escapingdense}]
    Let $w_0\in\C$ be arbitrary, and consider its orbit $(w_n)_{n\geq 0}$, i.e.\ 
   $w_n=f^n(w_0)$.  
    Let $D$ be a arbitrary disc around
    $w_0$; it is enough to show that $f^n(D)\cap I(f)\neq \emptyset$ for
     some $n$. So assume, by contradiction,
     that $f^n(D)\cap I(f)=\emptyset$ for all $n$.
    Since $[0,\infty)\subset I(f)$,
    we then have
        \[D_n \defeq  f^n(D)\subset U =
      \mathbb{C}\setminus[0,\infty) \qquad \text{for all $n\geq 0$.}\]

  By Pick's theorem \ref{thm:pick}, the hyperbolic derivative of $f^n$,
    as a map from $D$ to $U$, satisfies
        \[ \delta_n \defeq  \|\Deriv f^n(w_0)\|_D^U =
        |(f^n)'(w_0)|\cdot \frac{\rho_U(w_n)}{\rho_D(w_0)}\leq 1. \]
    Set $\eta_n\defeq \|\Deriv f(w_n)\|_U^U $; then $\delta_{n+1}=\delta_n\cdot\eta_n$ for all $n\geq 0$.

  By Lemma \ref{lem:hyper}, we know that $\eta_n>1$. Hence, for the
   sequence $\delta_n$ to remain bounded, we must necessarily have
   $\eta_n\to 1$. Indeed,
    $(\delta_n)_{n\geq1}$ is an increasing and bounded, and hence convergent, sequence.
    Thus
        \[ \lim_{n\rightarrow\infty}\eta_n= \lim_{n\rightarrow\infty}\frac{\delta_{n+1}}{\delta_{n}}=
        \frac{\displaystyle{\lim_{n\rightarrow\infty}\delta_{n+1}}}{
        \displaystyle{\lim_{n\rightarrow\infty}\delta_{n}}}=
        1.  \]

   By Lemma \ref{lem:hyper}, we see that
   \begin{equation} \label{eqn:argument}
      \min\bigl(|w_n|,\Arg(w_n)\bigr)\to 0,
    \end{equation} and hence
    all finite accumulation points of  $(w_n)$ must
     lie in the interval $[0,\infty)$.
    Since $w_0$ is not in the escaping set,
    the set of such finite accumulation points is nonempty.

    As this set is closed, we can let
    $z_0\in [0,\infty)$ be the smallest finite accumulation point
    of the sequence $(w_n)$; say $w_{n_k}\to z_0$.
   Notice that $\re w_{n_{k-1}}=\log|w_{n_k}|\rightarrow\log z_0 <z_0$. By
    choice of $z_0$, the sequence $w_{n_{k-1}}$ cannot have a finite
    accumulation point, and hence $|\im w_{n_{k-1}}|\to \infty$. This
    contradicts~\eqref{eqn:argument}, and we are done.
 \end{proof}

 The proof leaves open the possibility that the escaping set has nonempty
  interior (or even, a priori, that $I(f)=\C$). We shall now exclude this possibility,
  which then allows us to establish sensitive dependence on initial conditions. 

 \begin{thm}[Preimages of the negative real axis]\label{thm:negativereal}
    Let $W\subset\C$ be open and nonempty. Then there are infinitely
    many numbers $n\geq 0$ such that $f^n(W)\cap (-\infty,0]\neq \emptyset$.
 \end{thm}
\begin{remark}
  This will also follow from the (stronger) results in the next section,
   but the proof we give here relies on the same ideas as that
   of Theorem \ref{thm:escapingdense}, which gives the argument a nice
   symmetry.
\end{remark}
 \begin{proof}
  Let $D\subset W$ be an open disc. We shall first prove that there
   is at least one $n$ such that $f^n(D)$ intersects the negative real axis.
   So assume, by contradiction, that
         \[ f^n(D) \subset \tilde{U} \defeq  \C\setminus (-\infty,0] \]
   for all $n\geq 0$. By Theorem \ref{thm:escapingdense}, there is a
   point $w_0\in D\cap I(f)$.

  We proceed similarly as in the proof of Theorem \ref{thm:escapingdense},
   but now use the
   the hyperbolic metric of $\tilde{U}$.
   The set $\tilde{U}\defeq \mathbb{C}\setminus(-\infty,0]$ is not
   backward invariant, hence $f$ is not locally expanding
   at every point of $\tilde{U}$. However,
   there \emph{is} expansion~-- even strong expansion~--
   at points with sufficiently large real parts:

 \begin{claim}
   If $\zeta\in f^{-1}(\tilde{U})$ and $\re \zeta \geq 2$, then
    \[ \|\Deriv f(\zeta)\|_{\tilde{U}}^{\tilde{U}} \geq \sqrt{2}. \]
 \end{claim}
 \begin{proof}
   This follows by a similar calculation as in the proof of
    Lemma \ref{lem:hyper}: if we write
     $\zeta=re^{i\theta}$ with $\theta\in(-\pi/2,\pi/2)$,
    then $f(\zeta)=e^{re^{i\theta}}=e^{r(\cos\theta+i\sin\theta)}$, and
    \begin{align*}
\|\Deriv f(\zeta)\|_{\tilde{U}}^{\tilde{U}}
     &=|f'(\zeta)|\cdot\frac{2|\zeta|\cdot|\cos\left(\frac{1}{2}\Arg(\zeta)\right)|}{
          2|e^{\zeta}|\cdot\bigl|\cos\bigl(\frac{1}{2}\Arg\bigl(e^{\zeta}\bigr)\bigr)\bigr|}  \\
   &=
          \frac{r\cdot\cos(\theta/2)}{\bigl|\cos\bigl(\frac{1}{2}\Arg\bigl(e^{\zeta}\bigr)\bigr)\bigr|}
          \geq r\cdot\cos(\theta/2) \geq r\cdot \frac{\sqrt{2}}{2} \geq
           \sqrt{2}. \qedhere
   \end{align*}
 \end{proof}

 Now the proof proceeds along the same lines as before. Set
  $w_n \defeq  f^n(w_0)$ and $D_n \defeq  f^n(D)$ for $n\geq 0$. Since
   $f^n\colon D\to \tilde{U}$ is a holomorphic map, we again have
     \[ \delta_n \defeq 
             \|\Deriv f^n(w_0)\|^{\tilde U}_D \leq 1 \]
   for all $n$ by Pick's Theorem. The numbers
    $\eta_n\defeq \|\Deriv f(w_n)\|_{\tilde U}^{\tilde U}=\frac{\delta_{n+1}}{\delta_n}$ need not be
   bounded below by $1$. However, since $w_0\in I(f)$, there is
   $n_0$ such that $\re w_n \geq 2$ for $n\geq n_0$, and hence
  \[ \delta_n = \delta_{n_0}\cdot \|\Deriv f^{n-n_0}(w_{n_0})\|_{\tilde{U}}^{\tilde{U}} \geq
            \delta_{n_0}\cdot 2^{(n-n_0)/2}\to \infty. \]
  This is a contradiction.

  So \lrg{$f^n(W)\cap (-\infty,0]\neq\emptyset$} for some
   $n$. To see that there are infinitely many such $n$, set $k_0=n$,
   and apply
   the result to $f^{k_0+1}(W)$ to find some $k_1>k_0$ with
   $f^{k_1}(W)\cap (-\infty,0]\neq \emptyset$. Proceeding inductively, we find
   an infinite sequence $(k_j)$ with the desired property.
 \end{proof}

 \begin{cor}[Sensitive dependence] \label{cor:sensitive}
    The exponential map $f\colon \C\to\C$ has sensitive dependence on initial
    conditions with respect to spherical distance.
 \end{cor}
 \begin{proof}
 We shall prove that $f$ satisfies Definition \ref{defn:sphericalsensitive} with $R=1$ and $\delta=1$. Indeed, let 
   $U\subset\C$ be open. Then, by Theorem \ref{thm:escapingdense}, $U$ contains an escaping point $w$, and 
    there
    is some $n_0$ such that $|f^{n}(w)|\geq 2$ for all $n\geq n_0$.
   By Theorem \ref{thm:negativereal}, there is $n_1>n_0$ and some point
   $z\in D$ such that $f^{n_1}(z)\in (-\infty,0]$.
   Hence, for $n \defeq  n_1 + 1$, we have
   $|f^n(z)|\leq 1=R$ and $|f^n(w)-f^n(z)|\geq |f^n(w)|-|f^n(z)|\geq 1=\delta$, as desired.
 \end{proof}

\section{Transitivity and dense orbits}\label{sec:transitivity}

 Recall that \emph{topological transitivity}, one of the properties required in the
  definition of chaos, means that we can move between any two
  nonempty open subsets of the complex plane by means of the iterates of $f$. 
  The goal of this section is to establish transitivity, and deduce 
  that there are also points with dense orbits. 

 \begin{thm}[Topological transitivity] \label{thm:trans}
    If $U$, $V$ are nonempty and open, then there exists $n\geq0$ such that $f^n(U)\bigcap V\neq\emptyset$.
 \end{thm}
 We shall use the fact\lrg{, established in the previous section,} that the escaping set is dense in the plane. 
  The key point \lrg{is} that 
   $f$ is strongly expanding along the orbit of any escaping point $z_0$. 
   Hence, for sufficiently large $n$, 
    $f^n$ maps a small disc around $z_0$ to a set that contains
   a disc of radius $2\pi$ centred at
    $f^n(z_0)$. \lrg{By elementary mapping properties of the exponential map, the latter disc is spread, after two more applications of $f$, 
   over a large part of the   
   complex plane}. This establishes topological transitivity. 

 We now provide the details of this argument. Let us begin with a simple observation.
  \begin{obs}[Expansion along escaping orbits]\label{obs:escapingexpansion}
   Let $z_0\in I(f)$, and set $z_n \defeq  f^n(z_0)$ for $n\geq 1$. 
    Then $\re z_n\to\infty$ and 
   $|(f^n)'(z_0)|\to\infty$ as $n\to\infty$.
  \end{obs}
 \begin{proof}
   Since $z_{n+1}=e^{z_n}$ for all $n\geq 0$, we have
      \[ \re z_n = \log |z_{n+1}| \to \infty \]
    by definition of $I(f)$. 

   Furthermore, $|f'(z_n)| = |f(z_n)|=|z_{n+1}|\to \infty$, and hence there is 
   $N\in\N$ such that $|f'(z_n)|\geq2$ for all $n\geq N$. For $m> N$, we can use the chain rule to compute the derivative of
    $f^m = f \circ f\circ \dots \circ f \circ f^N$: 
   \[ |(f^m)'(z_0)| = \bigl|\bigl(f^N\bigr)'(z_0)\bigr|\cdot \prod_{n=N}^{m-1}|f'(z_n)| \geq
                                2^{m-N} \cdot \bigl|\bigl(f^N\bigr)'(z_0)\bigr| \to \infty \]
    as $m\to\infty$. Here we used the fact that $|(f^N)'(z_0)|\neq 0$, since $f'(z)\neq 0$ for
    all $z\in\C$.
 \end{proof}

We next prove the above-mentioned fact concerning the iterated 
  images of small discs around escaping points.

 \begin{prop}[Small discs blow up] \label{prop:trans}
    Let $z_0\in I(f)$. For $n\geq 1$, set $z_n\defeq f^n(z_0)$ and consider
    the disc
    $D_n$ of radius $2\pi$ centred at $z_n$; i.e.\ $D_n\defeq D_{2\pi}(z_n)$.
    Then there are $n_0\in\N$ and a sequence $(\phi_n)_{n\geq n_0}$ of holomorphic maps 
       $\phi_n\colon  D_n\rightarrow\mathbb{C}$ with the following properties:
    \begin{enumerate}[{\normalfont(a)}]
        \item $\phi_n(z_n)=z_0$, \label{item:correctbranch}
        \item $f^n(\phi_n(z))=z$ for all $z\in D_n$, \label{item:inversebranch}
        \item $\sup_{z\in D_n}|\phi'_n(z)|\to 0$ as $n\to\infty$, and \label{item:transcontraction}
       \item $\diam(\phi(D_n))\to 0$ as $n\to\infty$. \label{item:smalldiameter}
    \end{enumerate}
   (That is, for large $n$ there is
    a branch  $\phi_n$ of $f^{-n}$ that takes $z_n$ back to $z_0$ and 
    is uniformly strongly contracting.)
 \end{prop}
 \begin{proof}
  Observe that \ref{item:transcontraction} implies \ref{item:smalldiameter}. Indeed, for $z\in D_n$,
   \begin{equation}\label{eqn:meanvalue}
     |\phi_n(z) - \phi_n(z_n)| \leq |z-z_0|\cdot \sup_{\zeta \in D_n}|\phi_n'(\zeta)| \leq
          2\pi \sup_{\zeta \in D_n}|\phi_n'(\zeta)|
   \end{equation}
   by the
   mean value inequality. Hence by \ref{item:transcontraction},
    \[ \diam(\phi(D_n)) \leq 
       4\pi \sup_{\zeta \in D_n} |\phi_n'(\zeta)| \to 0 \qquad \text{as $n\to\infty$.}\]

 To prove the proposition, first assume additionally that 
    $|z_n|\geq2\pi+2$ for all $n\geq1$.
    In this case, none of the discs $D_n$ contain the origin.
    For each $n\geq 1$, \lrg{let}
    $L_n\colon D_n\rightarrow\mathbb{C}$ be the branch of the logarithm
    with $L_n(z_n)=L_n(e^{z_{n-1}})=z_{n-1}$.
    What can we say about the range of this map $L_n$? 
    Since $|z_n|\geq2\pi+2$ for all $n\geq1$, we have $|z|\geq 2$ for all $z\in D_n$, and hence
           \[|L_n'(z)|=1/|z|\leq 1/2. \]
    Again using the mean value inequality, we see that     \begin{align}
        |L_n(z)-z_{n-1}|=|L_n(z)-L(z_n)|\leq\sup_{\zeta\in D_n}|L'(\zeta)|\cdot \pi
    \end{align}
  for each $n$ and
     all $z\in D_n$. 
     In particular, and importantly, $L_n(D_n)\subset D_{n-1}$.

 It follows by induction that the composition 
   $\phi_n \defeq  L_1 \circ L_2 \circ \dots \circ L_n$ is defined on $D_n$\lrg{, 
    with $|\phi_n'(z)|\leq 2^{-n}$ for $z\in D_n$.}
  Hence $\phi_n$ satisfies~\ref{item:transcontraction}, and 
   \lrg{\ref{item:correctbranch} and~\ref{item:inversebranch} hold} by construction.
   This  completes \lrg{the proof} when $|z_n|\geq 2\pi +2$ for all $n\geq 1$.

 If this is not the case, then~-- since $z_0$ is an escaping point~-- there still exists 
    $n_1$ such that 
    $|z_n|\geq2\pi+2$ for all $n> n_1$. We can thus apply the preceding case 
    to the point $z_N$. This means that, for every $n>n_1$, there is a holomorphic map
     $\psi_n\colon D_n\to D_{n_1}$ such that 
        \begin{enumerate}[(a)]
            \item $\psi_n(z_n)=z_m$,
            \item $f^{n-n_1}(\psi_n(z))=z$ for all $z\in D_n$,
            \item $\sup_{z\in D_n}|(\psi_n)'(z)|\to0$ as $n\to\infty$, and
            \item $\diam(\psi_n(D_n))\to 0$ as $n\to\infty$.
        \end{enumerate}
 By the inverse function theorem, there exists a neighbourhood
  $U$ of $z_{n_1}$ and a branch $\pi\colon U\to\C$ of $f^{-n_1}$ mapping $z_{n_1}$ 
  to $z_0$. (This also follows from repeated applications of \lrg{suitable branches of the logarithm.} Observe that $z_n = f(z_{n-1})\neq 0$ for $n\geq 1$.)

   Now let $K$ be a small closed disc around $z_{n_1}$ with $K\subset U$. By
    \ref{item:smalldiameter}, we have $\psi_n(D_n)\subset K$ for sufficiently
    large $n$~-- say, for $n\geq n_0$. Hence we can define
     \[ \phi_n \colon D_n \to \C; \quad \phi_n(z) \defeq  \pi(\psi_n(z)). \]
    This map satisfies~\ref{item:correctbranch} and~\ref{item:inversebranch} by
    definition, and
    \[ \sup_{z\in D_n}|\phi_n'(z)| = \sup_{z\in D_n} \bigl(|\psi_n'(z)|\cdot |\pi'(\psi_n(z))|\bigr) \leq
                \max_{z\in D_n} |\psi_n'(z)|\cdot \sup_{w\in K}|\pi'(w)| \to 0 \]
    by property \ref{item:transcontraction}) of $\psi_n$. (Since $K$ is compact, the continuous function $|\pi'|$
     assumes its maximum on $K$, which is independent of $n$.)
 \end{proof}

\lrg{As mentioned above, after two additional iterates 
  the discs in the preceding Proposition will cover
  a large portion of the plane, as long they lie far enough to the right:}
\begin{obs}[Images of large discs] \label{obs:montel}
  Let $K$ be any nonempty compact subset of the punctured plane $\C\setminus\{0\}$. Then there 
   is $\rho>0$ with the following property. Suppose that $D$ is a disc of radius $2\pi$,
   centred at a point \lrg{having} real part at least $\rho$. Then
     $K\subset f^2(D)$.
\end{obs}
\begin{proof}
The disc $D$ contains
   a closed square of side-length $2\pi$, also centred at $\zeta$. Let $a=\re\zeta - \pi$ be the
    real part of the left vertical edge of $S$, then $a+2\pi=\re\zeta+\pi$ is the real part of the right vertical edge of $S$. 
  What is the image of $S$ under $f$? Looking back at Figure \ref{fig:exponential}, we see that
   it is precisely a closed round annulus $A$ around the origin, with inner radius
    $r_- = e^a$ and outer radius $r_+ = e^{a+2\pi}$. 

\begin{figure}%
\def\svgwidth{.48\textwidth}%
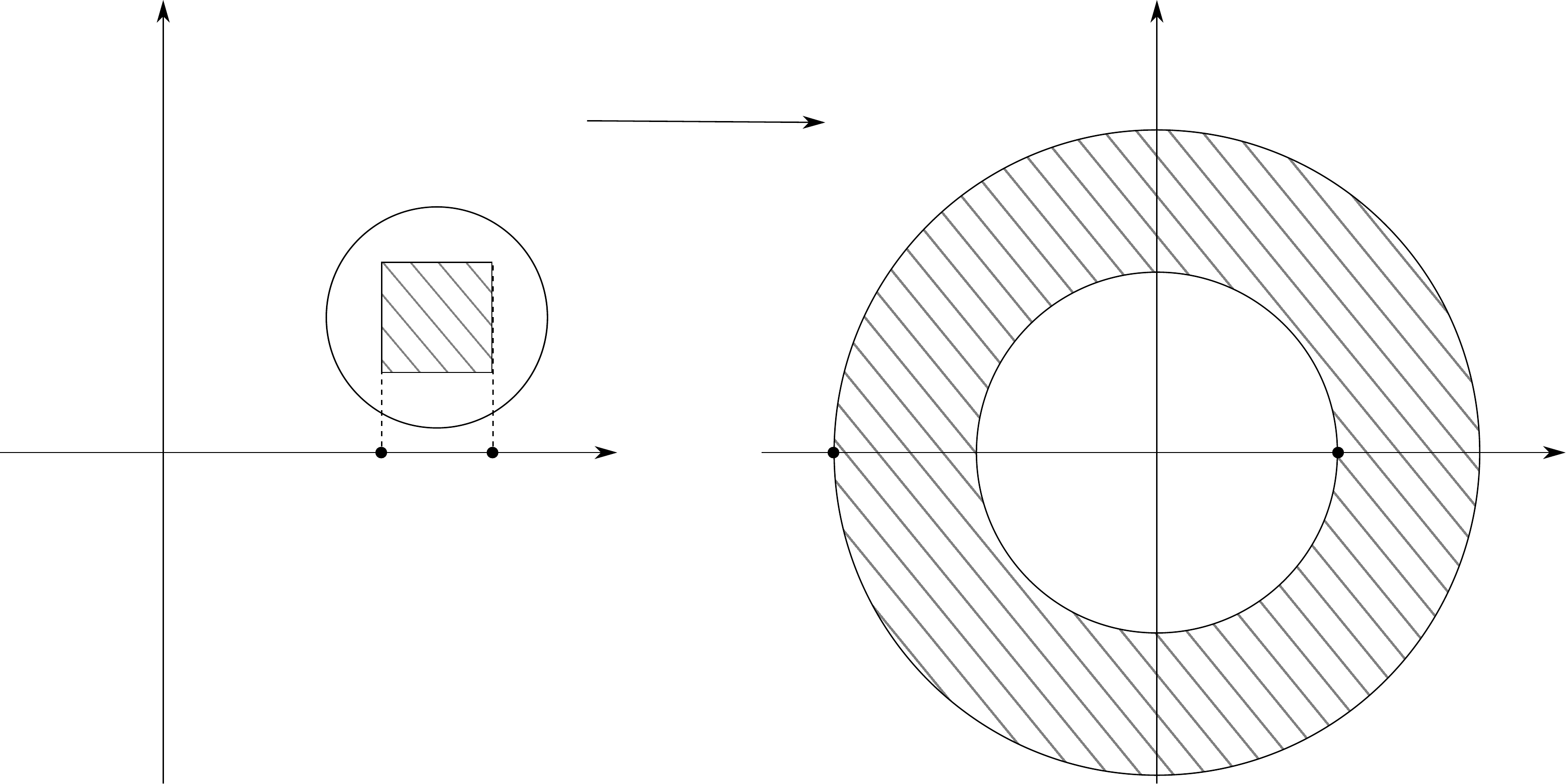
\hfill
\def\svgwidth{.48\textwidth}%
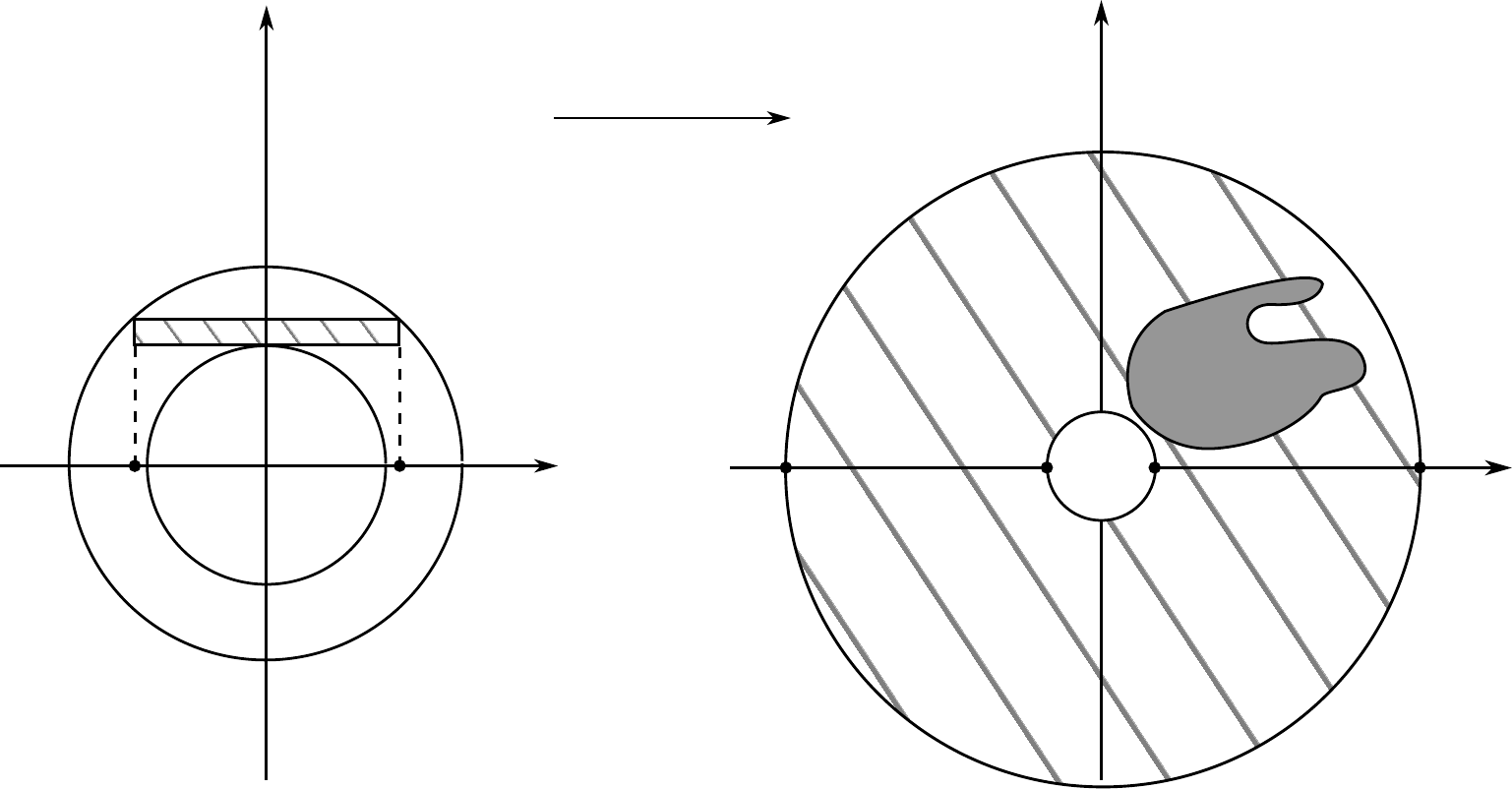
 \caption{\label{fig:transitivity}In the proof of Observation \ref{obs:montel}, the square
   $S$ is mapped to a large annulus $A$. The rectangle $R$ shown in $A$ maps to an annulus of huge outer
   radius and tiny inner radius, containing the compact set $K$.}
\end{figure}

  If $a$ is sufficiently large, then $A$ is a rather thick annulus with large inner radius. It
   follows that $A$ contains a long rectangular strip of height $2\pi$, and the image of this
   strip will cover most of the complex plane, including the compact set $K$. More
   precisely, suppose that $a\geq 0$; then $r_+-r_- \geq e^{2\pi} > 2\pi$. Hence we 
   can fit a maximal rectangle $R$ of height $2\pi$ into $A$, 
   symmetrically with respect to the imaginary axis and tangential to the inner boundary
   circle of $A$. (See Figure \ref{fig:transitivity}.)

  Let $\ell$ be the maximal real part of $R$ (i.e., $\ell$ is half the horizontal side-length of
   $R$). We can compute $\ell$ using the Pythagorean Theorem:
        \[
            \ell^2=r_+^2 - (r_-+2\pi)^2 > r_+^2 - 4r_-^2 =
                                                 (e^{4\pi} - 4)\cdot e^{2a} > e^{2a}, \]
     provided that $r_-\geq 2\pi$. Hence we see that $\ell > e^a$, and the image of
     $R$ includes all points of modulus between $e^{-\ell}$ and $e^{\ell}$. 

     In particular, we can set
      \[ \rho \defeq  3\pi + \max_{z\in K} \bigl|\log|z|\bigr|. \]
    If $\re \zeta \geq \rho$, then we have $a \geq 2\pi$, and 
     $\ell > e^a > a > \bigl|\log|z|\bigr|$. The claim follows.
\end{proof}

As a consequence, we obtain the following stronger version of Theorem \ref{thm:trans}.

\begin{cor}[Open sets spread everywhere] \label{cor:montel}
 Let $K\subset \C\setminus\{0\}$ be compact, and let $U\subset\C$ be open and nonempty.
   Then there is some $N\in\N$ such that
    $K\subset f^n(U)$ for all $n\geq N$.
\end{cor}
\begin{remark}
  To deduce the original statement of Theorem \ref{thm:trans}, simply choose $K$ to consist
   of a single point in $V$. 
\end{remark}
\begin{proof}
   Since the escaping set is dense in the plane, there exists some $z_0\in I(f)\cap U$.  
    Set $z_n \defeq  f^n(z_0)$ for $n\geq 1$, and let $D_n \defeq  D_{2\pi}(z_n)$, $n_0$ and
    $\phi_n$ be as in Proposition \ref{prop:trans}. Also let $\rho$ be the number from
   Observation \ref{obs:montel} (for the same set $K$). Since $z_0$ is an escaping point,
   and by part~\eqref{item:smalldiameter} of
     Proposition \ref{prop:trans},
    we can choose $N_1\geq n_0$ such that
    $\re z_n \geq \rho$ and $\phi_n(D_n)\subset U$ for all $n\geq N_1$. 

   Let $n\geq N_1$. Then 
    \[ f^{n+2}(U) \supset f^{n+2}(\phi_{n}(D_n)) = f^2(f^n(\phi_n(D_n))) =
         f^2(D_n)\supset K \]
    by Observation \ref{obs:montel}. Hence the claim holds with $N \defeq  N_1+2$.
 \end{proof}

 The existence of dense orbits is closely related, and
  often equivalent, to that of topological transitivity. Indeed, we can
  deduce the former from Theorem \ref{thm:trans}, by using the 
    \emph{Baire category theorem} \cite[Exercise 16 in Chapter 2]{babyrudin}. 
  This theorem implies that any countable intersections of open and dense subsets of the 
   complex plane is itself dense and uncountable. 

 \begin{cor}[Dense orbits] \label{cor:denseorbit}
  The set $\mathcal{T}$ of all points $z_0\in\C$ with dense orbits under the exponential map 
    is uncountable and dense in $\C$. 
 \end{cor}
 \begin{proof}
  Let $D\subset\mathbb{C}$ be any non-empty open set. \lrg{By Theorem \ref{thm:trans},}
    the inverse orbit
   \[ \Orbit^-(D) \defeq  \bigcup_{n\geq 0} f^{-n}(D) \]
   is a dense subset of $\C$. Note that $\Orbit^-(D)$ is also open, as a union of open subsets. 

  Now consider the countable collection 
    \[ \mathcal{D} \defeq  \{ D_r(z) \colon  \re z, \im z, r\in \Q \} \] of 
   open discs with rational centres and radii. 
   Clearly $z_0\in \mathcal{T}$ if and only if the orbit of $z_0$ enters every element of 
   $\mathcal{D}$ at least once, i.e.
   \[ \mathcal{T} = \bigcap_{D\in\mathcal{D}} \Orbit^{-}(D). \]
  Hence $\mathcal{T}$ is indeed uncountable and dense, as a countable 
    intersection of open and dense subsets of $\C$. 
 \end{proof}

\section{Density of periodic points}\label{sec:periodicpoints}

We are now ready to complete the proof that the exponential map is chaotic, by proving 
  density of the
   set of periodic
   points. In fact, we shall prove slightly more, namely that \emph{repelling} periodic points
   of $f$ are dense in the plane. Here a periodic point $p$, with $f^n(p)=p$, is called 
   \emph{repelling} if $|(f^n)'(p)|>1$. This ensures that any point $z\neq p$
   close to $p$ is (initially) ``repelled'' away from the orbit of $p$~-- density
   of such points gives 
   another indication that the dynamics of the exponential map is highly unstable! (It is, however,  
   not difficult \lrg{to} show directly that \emph{all} periodic points of $f$ are repelling; see 
   Exercise \ref{ex:allperiodicpointsrepelling}.)

\begin{thm}[Density of repelling periodic points] \label{thm:repel}
  Let $U\subset\C$ be open and nonempty. Then there exists a repelling periodic point
    $p\in U$. 
\end{thm}

The idea of the proof is, again, to begin
    with an escaping point $z_0$. Our goal is to find an inverse branch
   of an iterate of $f$ that maps a neighbourhood of $z_0$ back into itself, with strong 
   contraction. Then the existence of a periodic point follows from the contraction mapping theorem 
   \cite[Theorem 3.6]{burkillburkill} (also known as the \emph{Banach fixed point theorem}). 
   We noticed already in the last section 
   that the disc $D_n$ of radius $2\pi$ around
   the $n$-th orbit point $z_n$ can be pulled back along the orbit in one-to-one fashion (Proposition
   \ref{prop:trans}). We need to be able
   to ``close the loop'', by pulling  back a small disc around $z_0$ into
   $D_n$. In other words, we are looking for a more precise version of
   Observation \ref{obs:montel}, as follows. 

\begin{lem}[Inverse branches of $f^2$] \label{lem:periodic}
 Let $z_0\in\mathbb{C}\setminus\{0\}$. Then there are a disc $\Delta$ centered at 
   $z_0$ and a number $\rho\geq 0$ with the
   following property: 

For any disc $D$ of radius $2\pi$ centred at 
    a point with real part at least $\rho$, there is  
      $\phi\colon  \Delta\rightarrow D$ such that 
      $f^2(\phi(z))=z$ and $|\phi'(z)|\leq 1$ 
      for all $z\in \Delta$. 
\end{lem}
\begin{remark}
  In other words, there exists a disc $\Delta$ around $z_0$ that is not only covered by
    $f^2(D)$ (as we know it must be from Observation \ref{obs:montel}), but on which we can
   even define a branch of $f^{-2}$ that takes it back into $D$. 
   It is not difficult to extend the result, with a similar proof, to see that this
   is true for \emph{any} disc $\Delta$~-- and indeed any simply-connected domain~-- whose
   closure 
   is bounded away from $0$ and $\infty$.
\end{remark}
\begin{proof}
 Set $\eps \defeq \frac{|z_0|}{2}$ and $\Delta \defeq  D_{\eps}(z_0)$. \lrg{For any $n\in\Z$, there is a branch 
    $L_n$ of the logarithm}
    on $\Delta$ whose values have imaginary parts between
    $\Arg z_0+(2n-1)\pi$ and $\Arg z_0+(2n+1)\pi)$. Each $L_n$ is continuous on $\Delta$ 
     and satisfies $f(L_n(z))=z$ for all $z\in \Delta$. 

 Consider the sets $V_n \defeq  L_n(\Delta)$, with $n\in\Z$; then $V_n = V_0 + 2\pi i n$.
   So the $V_n$ form a linear sequence of domains
   of uniformly bounded diameter, tending to infinity in the direction of the positive and
   negative imaginary axes. Consider all preimage components of some $V_n$ under $f$, 
   for $n\neq 0$ (since $V_0$ might contain the origin). On each of these,
   we can define a branch of $f^{-2}$ taking values in $\Delta$~-- hence we should show that
   any disc $D$ of radius $2\pi$ contains at least one such component, provided that its centre
   lies sufficiently far to the right. This should be clear from the mapping behaviour of the
   exponential map (Figure \ref{fig:exponential}). Indeed, for every odd multiple of $\pi/2$, there
   is a sequence of sets in question whose imaginary parts tend to this value, and having 
   real parts closer and closer
   together (see Figure \ref{fig:preimage}). 

\begin{figure}
\begin{center}%
\def\svgwidth{\textwidth}%
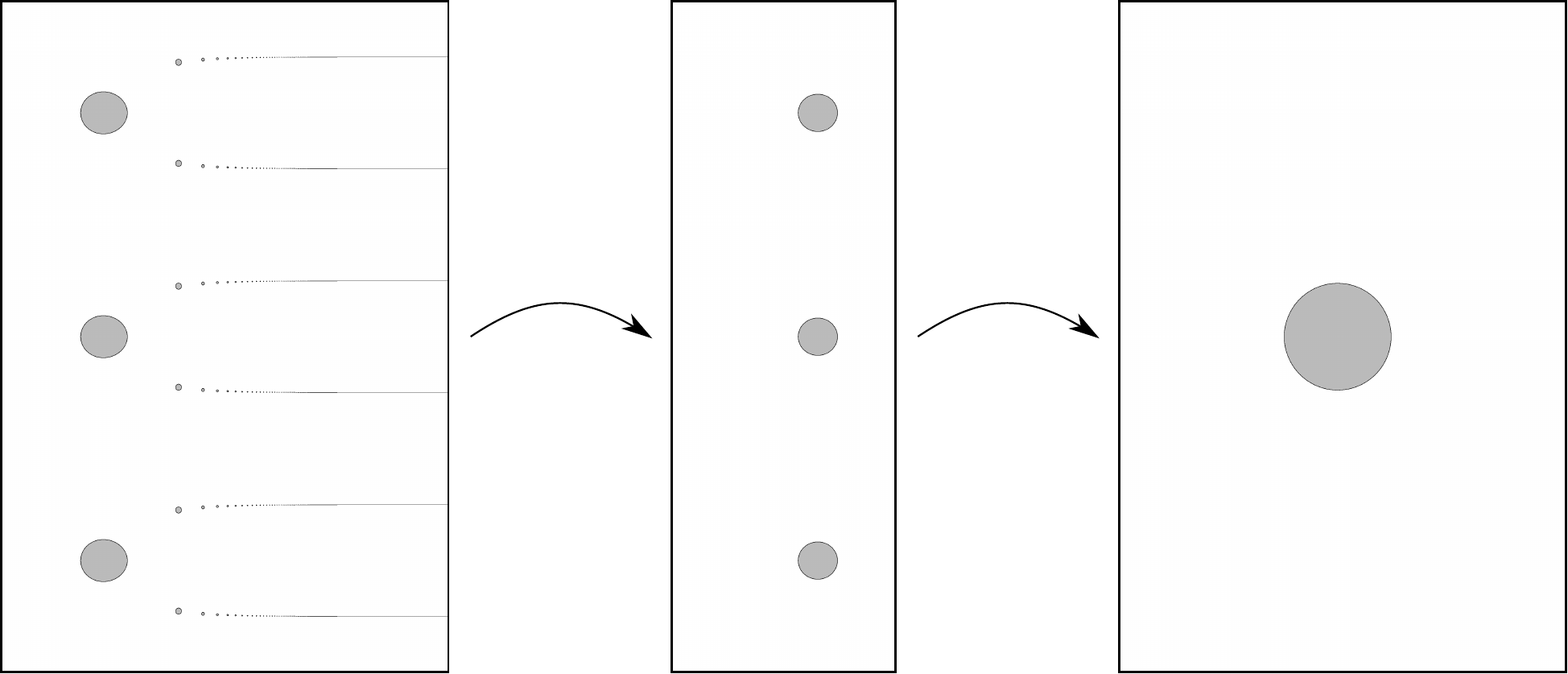
\end{center}
 \caption{\label{fig:preimage}This is an illustration of the proof of Lemma \ref{lem:periodic}. Shown are the
   first and second preimage component of the disc $\Delta = D_{\frac{3}{2}}(3)$.}
\end{figure}

  More formally, consider the points $w_n \defeq  L_n(z_0)\in V_n$. Set $\rho_1 \defeq  |w_0|+3$ and 
   let $D$ be a disc of radius $2\pi$ centred at a point $\zeta$ with $\re \zeta \geq \rho_1$. 
   Since $w_n = w_0 + 2\pi i n$, we can find $n_0\geq 1$ such that 
     \begin{equation}\label{eqn:sizeofwn}
       |e^{\zeta}|-\pi \leq |w_{n_0}| \leq |e^{\zeta}|+\pi.\end{equation}
     Furthermore, choose $m$ such that 
     $|\im \zeta - (4m+1)\pi/2|\leq \pi$\lrg{, and let $\psi$ be the branch of the logarithm mapping the 
     upper half-plane to the strip} $\{z\in\C\colon  2m\pi < \im z < (2m+1)\pi\}$.

  Let $w\in V_{n_0}$. Then $|\im\psi(w) - \im \zeta|\leq 3\pi/2$ by choice of $\psi$. Furthermore,
   by definition of $\Delta$ and $L_n$, we have 
    \[ |\re w - \re w_{n_0}| = \left|\log \frac{|w|}{|w_{n_0}|}\right| \leq \log 2. \]
   By~\eqref{eqn:sizeofwn}, we conclude that
     \[    |e^{\zeta}| - 2\pi - \log2 \leq |w| \leq |e^{\zeta}| + 2\pi + \log2. \]
   Dividing by $|e^{\zeta}|$, and recalling that $\re \zeta \geq \rho_1 > 3 > \log 5\pi$, we see that 
    \[ \frac{1}{2} < 1 - \frac{2\pi + \log 2}{|e^{\zeta}|} < 
          \frac{|w|}{|e^{\zeta|}} < 1 + \frac{2\pi + \log 2}{|e^{\zeta}|} < 2. \]
  Hence 
    \[ |\re \psi(w) - \re \zeta| = |\log|w| - \re \zeta| = 
         \left|\log\frac{|w|}{|e^{\zeta}|}\right|
         <  \log 2. \]
   Thus
      \[ |w - \zeta| < 3\pi/2 + \log2 < 2\pi, \]
   and hence $\psi(V_n)\subset D$. Hence the branch $\phi \defeq  \psi\circ\phi_n$ indeed
    maps $\Delta$ into $D$. 
   Moreover, 
         \[ \phi_n'(z) = \frac{1}{|z|} \leq \frac{2}{|z_0|} \]
     for all $z\in \Delta$, while 
           \[ |\psi'(z)|=\frac{1}{|z|} \leq \frac{2}{e^{\zeta}} \] on $V_n$. Hence, if
    we choose $\rho \defeq  \max(\rho_1, \log 4 - \log|z_0|)$, 
     then $|\phi'(z)|\leq 1$ for all $z\in \Delta$, as required. 
\end{proof}

\begin{proof}[Proof of Theorem \ref{thm:repel}] 
 Let $U\subset \C$ be open and nonempty. By Theorem \ref{thm:escapingdense}, there is 
  an escaping point $z_0\in U\setminus\{0\}$. Choose a disc
    $\Delta$ around $z_0$ and $\rho>0$ and as in Lemma
   \ref{lem:periodic}. By shrinking $\Delta$, if necessary, we may suppose that $\Delta\subset U$.

  Let $\Delta_1$ be a smaller disc also centred at $z_0$. so that 
    $\overline{\Delta_1}\subset \Delta$. As in Section \ref{sec:transitivity}, 
   set $z_n\defeq f^n(z_0)$ and $D_n\defeq D_{2\pi}(z_n)$. Let 
    $(\phi_n)_{n\geq n_0}$ be the inverse branches from Proposition
    \ref{prop:trans}. By conclusions
    \ref{item:smalldiameter} and \ref{item:transcontraction} of that
   Proposition, we may assume that $n_0$ is large enough to ensure that 
    $\phi_n(z)\in \Delta_1$ and $|\phi_n'(z)|\leq 1/2$ for all $z\in D_n$. 

  Since $z_0$ is an escaping point, there is $N\geq n_0$ such that 
   $\re z_n \geq \rho$ for all $n\geq N$.
   By Lemma~\ref{lem:periodic}, there is thus a branch $\psi_n$ of $f^{-2}$ that maps
   $\Delta$ into $D_n$, with $|\psi_n'(z)|\leq 1$ for all $z\in\Delta$. 

  It follows that 
    $\phi_n (\psi_n(\overline{\Delta_1}))\subset \phi_n( D_n) \subset \Delta_1$, and that
     $\phi_n\circ \psi_n\colon  \overline{\Delta_1} \to \overline{\Delta_1}$ 
    is a contraction map. As $\overline{\Delta_1}$ is compact, and hence complete, 
    this function has a fixed point $p\in\overline{\Delta}$ by the contraction
    mapping theorem. By construction, 
     \begin{align*} f^{n+2}(p) &= f^{2}(f^n(\phi_n(\psi_n(p)))) = f^2(\psi_n(p)) = p \qquad
    \text{\lrg{and}} \\ 
     | (f^{n+2})'(p)| &= \frac{1}{|\psi_n'(p)|\cdot |\phi_n'(\psi_n(p))|} \geq 2. \end{align*}
    Hence $p$ is indeed a repelling periodic point of $f$, as required. 
\end{proof}

\section{Further results and open problems}\label{sec:further}
  The realization that the exponential map $f$ acts chaotically on the complex plane is not the end of the story. Rather, it leads to further 
    questions about the qualitative behaviour of $f$, and much research has been done 
    since Misiurewicz's \lrg{work}. The picture is still far from complete, and \lrg{several} interesting questions remain open. Here,
   we restrict to a small selection of results and ideas, referring to the literature for further information.

\smallskip

\noindent\textbf{The escaping set of the exponential.}
  The escaping set $I(f)$ of the exponential map played an important role in our proof of Theorems \ref{thm:orbits} and \ref{thm:chaotic}.
   We saw in Theorem \ref{thm:escapingdense} that this set is dense in the plane, and hence 
   it is plausible that a thorough study of its fine structure will yield information also about the non-escaping part of the dynamics. 
    We already \lrg{saw} that $I(f)$ contains the real axis together with all of its 
     preimages under iterates of $f$\lrg{-- but there are many other escaping points!} Indeed, 
     Devaney and Krych \cite{devaneykrych} observed the existence of uncountably many different curves to infinity in $I(f)$, most of which
     do \emph{not} reach the real axis under iteration. Later, Schleicher and Zimmer \cite{schleicherzimmer} were able to show that
     \emph{every} point of $I(f)$ can be connected to infinity by a curve in $I(f)$. 

  Maximal curves in the escaping set are  referred to as ``rays'' or ``hairs'', and \lrg{they provide a structure} that can be
    exploited in the study of the wider dynamics of $f$. However, \lrg{the way in which these rays 
    fit together} to form the entire escaping set is rather non-trivial. For example, \lrg{while each path-connected component of
    $I(f)$ is such a curve, and is relatively closed in $I(f)$ (i.e., rays do not accumulate on points that belong to other rays) the escaping set} 
     nonetheless turns out to be a \emph{connected} subset of the plane \cite{escapingconnected}. Furthermore,
    while many rays end at a unique point in the complex plane, some have been shown to accumulate everywhere upon themselves
     \cite{DevJar02,nonlanding}, 
     resulting in a very complicated topological picture. For which rays this can occur, and whether other types of
     accumulation behaviour are possible, requires further research.

  Taken together, these results provide some 
      indication that the escaping set is a rather complicated object. In fact, even iterated preimages of the negative real axis (all of which are simple curves
     tending to infinity in both directions) result in highly nontrivial phenomena and open questions. 
      In 1993, Devaney \cite{devaneyknaster} showed that the closure of a certain natural sequence of such preimages has 
    some  ``pathological'' topological properties. He also formulated a conjecture concerning the structure of this set and its stability under certain
      perturbations of the
     map $f$, which remains open to this day. 

\smallskip     

\noindent\textbf{Measurable dynamics \lrg{of the exponential}.}
 \lrg{Corollary \ref{cor:denseorbit} (and its proof) means that,
   topologically, ``most'' points have a dense orbit. It is natural
   to ask also about the behaviour of ``most'' points with respect to area:
   If we pick a point $z\in\C$ at random will its orbit
   be dense?} Lyubich \cite{lyubichexp} and Rees \cite{reesexp} independently 
   \lrg{gave an answer} in the 1980s: 
   For a random point $z\in\C$, the \lrg{orbit of $z$ is \emph{not} dense; rather, its set of limit points} coincides precisely
    with the orbit of $0$.
    In other words, after a certain number of steps, the
    orbit will come very close to $0$, and then follow the orbit
    $\{0,1,e,e^e,\dots\}$ for a finite number of steps. Our point might then
    spend some additional time close to $\infty$, 
    until it maps into the left half-plane. In the next step, it ends up even closer to $0$, 
    and so on. 
  We can furthermore also ask about the relative ``sizes'' of the \lrg{sets of points} with various other types
   of behaviour \lrg{(for example, with respect to \emph{fractal dimension}).
   It again turns out that
   there is a rich structure from this point of view; see~e.g.\ \cite{mcmullen,urbanskizdunikunstableexp,karpinskaurbanskiexp}.}

\smallskip 

\noindent\textbf{Transcendental dynamics.}
 To place the material in this paper in its proper context, let us finally discuss \lrg{iterating a holomorphic self-map}
    $f\colon \C\to\C$ of the complex plane in general. \lrg{To obtain interesting behaviour, we} assume that $f$ is non-constant and non-linear.
   As before, \lrg{a} key question is how the behaviour of orbits varies under perturbations of the starting point $z_0$. To this end, one divides the 
   starting values into two sets: The closed set of points near which there is sensitive dependence on initial conditions is called the
   \emph{Julia set} $J(f)$, while its complement~-- \lrg{where} the behaviour is \emph{stable}~-- is the
   \emph{Fatou set} $F(f)=\C\setminus J(f)$. \lrg{(See \cite{waltersurvey} for formal definitions and further background.)}
   By the magic of complex analysis, the~-- rather mild~--
   notion of instability used to define the Julia set always leads to globally chaotic dynamics:

 \begin{thm}[Chaos on the Julia set]\label{thm:fatoujuliabaker}
  The Julia set $J(f)$ is always uncountably infinite, and $f^{-1}(J(f))=J(f)$. Furthermore,
   the function $f\colon J(f)\to J(f)$ is chaotic in the sense of Devaney.
 \end{thm}

 The key part of this theorem is the density of periodic points in the Julia set. 
     \lrg{That 
     $J(f)$ is uncountable and that $f$ is topologically transitive
   was} known already (albeit in different terminology) to Pierre Fatou and (in the case of polynomials) Gaston Julia, who 
   independently founded
   the area of holomorphic dynamics in the early twentieth century. For polynomials~-- and indeed
   for rational functions~-- the density of periodic points in the Julia set was also established
   by Fatou and Julia, but it took about half a century until Baker \cite{bakerrepulsive}
   completed the proof for general entire functions. 

 With this general terminology, Misiurewicz's theorem says precisely that $F(\exp)=\emptyset$. (We emphasize 
   that there are many entire functions with nonempty Fatou set. As an example, \lrg{consider 
    $f(z)\defeq(e^z-1)/4$}. Then $f^n(z)\to 0$ for $|z|<1$, and hence the entire unit disc is in the Fatou set.)
   It can be shown that there are only
   a few possible types of behaviour for points in the Fatou set of an entire function \cite[\S~4.2]{waltersurvey}. Using classical methods, most 
   of these can be excluded fairly easily in the case of the
   exponential map; the difficult part is to show that there can be no
   \emph{wandering domain}, i.e.\ a connected component $U$ of the Fatou set such that $f^n(U)\cap U=\emptyset$ for all $n>0$. 
   Misiurewicz \lrg{used} the
   specific properties of the exponential function, but more general tools for ruling out the existence of wandering domains
   have since been developed.
  The most famous is due to Sullivan,
   who showed that rational functions never have wandering
   domains, answering a question left open by Fatou. Sullivan's argument, which uses deep results from complex analysis, can be
   extended to classes of entire functions containing the exponential map \cite{bakerrippon,alexmisha,goldbergkeen}. As mentioned in the introduction,
   this provides an alternative (though highly non-elementary) method  of
   establishing Misiurewicz's theorem. 
The study of wandering domains, and when they can occur,
    continues to be an active topic of research in transcendental dynamics; we refer to \cite{wandering} for a \lrg{discussion} and references. 

\section{Exercises}\label{sec:exercises}%
{\small%
\begin{exercise}[Topological transitivity and dense orbits] \label{ex:denseorbit}
  Let $X\subset \C$, and let $f\colon X\to X$ be a continuous function. We say that $x_0\in X$ \emph{has a dense orbit} if
   the set of accumulation points of the orbit of $x_0$  is dense in $X$. (Note that this differs subtly from the requirement that
   the orbit of $x_0$ is dense when thought of as a \emph{subset} of $X$. However, the two conditions are equivalent if $X$ has no
   isolated points.)

 Prove: if there is a point with a dense orbit under $f$, then $f$ is topologically transitive.
\end{exercise}

\begin{exercise}[Problems with using Euclidean distance in sensitive dependence] \label{ex:realexp}
 Show that the \emph{real} exponential map $f\colon \R\to\R; x\mapsto e^x$, exhibits sensitive dependence on initial conditions with respect to 
   Euclidean distance $d(x,y)=|x-y|$.
\end{exercise}

\begin{exercise}[Sensitive dependence for distance functions]\label{ex:sphericalsensitiveimpliesallothers}
  Let $f\colon \C\to\C$ be continuous and have sensitive dependence with respect to spherical distance.
 
  Prove: If $d\colon \C\times\C\to [0,\infty)$ is \emph{any}
     distance function that is topologically equivalent to Euclidean distance (i.e., a round open disc around a point $z_0$ contains some disc around $z_0$ in 
    the sense of the distance $d$, and vice versa),  then
    $f$ has sensitive dependence with respect to $d$.
    (So sensitive dependence with respect to spherical distance is the strongest such condition we can impose.)

  \emph{Hint.} Observe that the possible pairs $(f^n(z),f^n(w))$ in Definition \ref{defn:sphericalsensitive} all belong to a closed and bounded
     subset of $\C^2$, which depends only on $\delta$ and $R$. Then use the fact that $d$ is continuous as a function of two complex variables
     (with respect to Euclidean distance).
\end{exercise}

 \begin{exercise}[Automorphisms of the disc] \label{ex:mobius}
   Suppose that $f\colon \D\to\D$ is a conformal automorphism (i.e. holomorphic and bijective)
    with $f(0)=0$. Conclude, using the Schwarz lemma, that $f$ is a rotation around the origin.
    Deduce that any conformal automorphism $f\colon \D\to\D$ is of the form~\eqref{eqn:mobius}.
 \end{exercise}

\begin{exercise}[Expansion of the exponential map] \label{ex:derivativeestimate}
  Strengthen Lemma \ref{lem:hyper} by showing that
   $\|\Deriv f(\zeta_n)\|_U^U\to 1$
   if and only if $\Arg(\zeta_n)\to 0$ and $\im \zeta_n\to 0$. 
  (\emph{Hint.} Use the fact that $|\sin(x)|/|x|\to 1$ if and only if
    $|x|\to 0$.)
\end{exercise}
\begin{exercise}[Strong hyperbolic expansion for large real parts]
  For the domain $\tilde{U}$ in the proof of Theorem \ref{thm:negativereal},
   show that even
   $\|\Deriv f(\zeta)\|_{\tilde{U}}^{\tilde{U}}\to \infty$
  as $\re \zeta \to \infty$.

  (Either use a direct calculation as in the proof of the Claim, or
   a more conceptual explanation as in the first part of the proof of
   Lemma \ref{lem:hyper}.)
\end{exercise}
\begin{exercise}[Extremely sensitive dependence]
 Strengthen Corollary \ref{cor:sensitive} by showing that the exponential map satisfies 
   Definition \ref{defn:sphericalsensitive} for \emph{any} choice of $R,\delta>0$. 
\end{exercise}

\begin{exercise}[No nonrepelling orbits] \label{ex:allperiodicpointsrepelling}
     Use Lemma \ref{lem:hyper} to show that
      all periodic points of the exponential map are repelling. 

   (\emph{Hint.} Observe that, at a periodic point of period $n$, the 
    hyperbolic derivative of $f^n$ agrees precisely with the usual (Euclidean) derivative.)
 \end{exercise}}


\begin{thebibliography}{10}

\bibitem{ahlforscomplexanalysis}
L.~V. Ahlfors, \emph{Complex analysis}, third ed., McGraw-Hill Book Co., New
  York, 1978, International Series in Pure and Applied Mathematics.

\bibitem{andersonhyperbolic}
J.~W. Anderson, \emph{Hyperbolic geometry}, second ed., Springer
  Undergraduate Mathematics Series, Springer-Verlag London, Ltd., London, 2005.

\bibitem{bakerrepulsive}
I.~N. Baker, \emph{Repulsive fixpoints of entire functions}, Math. Z.
  \textbf{104} (1968), 252--256.

\bibitem{bakerrippon}
I.~N. Baker and P.~J. Rippon, \emph{Iteration of exponential functions}, Ann.
  Acad. Sci. Fenn. Ser. A I Math. \textbf{9} (1984), 49--77.

\bibitem{banksetal}
J.~Banks, J.~Brooks, G.~Cairns, G.~Davis, and P.~Stacey, \emph{On {D}evaney's
  definition of chaos}, Amer. Math. Monthly \textbf{99} (1992), no.~4,
  332--334.

\bibitem{beardonminda}
A.F. Beardon and D.~Minda, \emph{The hyperbolic metric and geometric function
  theory}, Quasiconformal Mappings and Their Applications (2007), 10--56.

\bibitem{waltersurvey}
W.~Bergweiler, \emph{Iteration of meromorphic functions}, Bull. Amer. Math.
  Soc. (N.S.) \textbf{29} (1993), no.~2, 151--188.

\bibitem{bergweileretalwandering}
W.~Bergweiler, M.~Haruta, H.~Kriete, H.-G.~Meier, and
  N.~Terglane, \emph{On the limit functions of iterates in wandering
  domains}, Ann. Acad. Sci. Fenn. Ser. A I Math. \textbf{18} (1993), no.~2,
  369--375.

\bibitem{burkillburkill}
J.~C.~Burkill and H.~Burkill, \emph{A second course in mathematical
  analysis}, Cambridge University Press, Cambridge-New York, 1980.

\bibitem{devaneyintroduction}
R.~L.~{Devaney}, \emph{{An introduction to chaotic dynamical systems}}, 2nd
  ed., Addison-Wesley Publishing Company, Inc., 1989.

\bibitem{devaneyknaster}
\bysame, \emph{Knaster-like continua and complex dynamics}, Ergodic
  Theory Dynam. Systems \textbf{13} (1993), no.~4, 627--634.

\bibitem{DevJar02}
R.~~L. Devaney and X.~Jarque, \emph{Indecomposable continua in
  exponential dynamics}, Conform. Geom. Dyn. \textbf{6} (2002), 1--12
  (electronic).

\bibitem{devaneykrych}
R.~L.~Devaney and M.~Krych, \emph{Dynamics of {${\rm exp}(z)$}},
  Ergodic Theory Dynam. Systems \textbf{4} (1984), no.~1, 35--52.

\bibitem{alexmisha1984kharkov}
A.~E. Eremenko and M.~Y. Lyubich, \emph{Iterates of entire functions},
  Preprint 6-84 of the Institute of Low Temperature Physica and Engineering,
  Kharkov, 1984, (Russian).

\bibitem{alexmisha}
\bysame, \emph{Dynamical properties of
  some classes of entire functions}, Ann. Inst. Fourier (Grenoble) \textbf{42}
  (1992), no.~4, 989--1020.

\bibitem{fatouentire}
P.~{Fatou}, \emph{{Sur l'it\'eration des fonctions transcendantes
  enti\`eres.}}, {Acta Math.} \textbf{47} (1926), 337--370.

\bibitem{fishercomplexvariables}
S.~D. Fisher, \emph{Complex variables}, Dover Publications, Inc., Mineola,
  NY, 1999, Corrected reprint of the second (1990) edition.

\bibitem{goldbergkeen}
Lisa~R. Goldberg and Linda Keen, \emph{A finiteness theorem for a dynamical
  class of entire functions}, Ergodic Theory Dynam. Systems \textbf{6} (1986),
  no.~2, 183--192.


\bibitem{karpinskaurbanskiexp}
B.~Karpi{\'n}ska and M.~Urba{\'n}ski, \emph{How points escape to
  infinity under exponential maps}, J. London Math. Soc. (2) \textbf{73}
  (2006), no.~1, 141--156.

\bibitem{lyubichexp}
M.~Y.~Lyubich, \emph{On the typical behaviour of trajectories of the
  exponent}, Russian Math. Surveys \textbf{41} (1986), no.~2, 207--208.

\bibitem{mcmullen}
C.~T.~McMullen, \emph{Area and {H}ausdorff dimension of {J}ulia sets of
  entire functions}, Trans. Amer. Math. Soc. \textbf{300} (1987), no.~1,
  329--342.

\bibitem{wandering}
H.~Mihaljevi{\'c}{-Brandt} and L.~Rempe{-Gillen}, \emph{Absence of
  wandering domains for some real entire functions with bounded singular sets},
  Math. Ann. \textbf{357} (2013), no.~4, 1577--1604.

\bibitem{misiurewiczexp}
M.~Misiurewicz, \emph{On iterates of {$e^{z}$}}, Ergodic Theory
  Dynamical Systems \textbf{1} (1981), no.~1, 103--106.

\bibitem{visualcomplexanalysis}
T.~Needham, \emph{Visual complex analysis}, The Clarendon Press, Oxford
  University Press, New York, 1997.

\bibitem{reesexp}
M.~Rees, \emph{The exponential map is not recurrent}, Math. Z. \textbf{191}
  (1986), no.~4, 593--598.

\bibitem{nonlanding}
L.~Rempe, \emph{On nonlanding dynamic rays of exponential maps}, Ann. Acad.
  Sci. Fenn. Math. \textbf{32} (2007), no.~2, 353--369.

\bibitem{escapingconnected}
\bysame, \emph{The escaping set of the exponential}, Ergodic Theory Dynam.
  Systems \textbf{30} (2010), no.~2, 595--599.

\bibitem{babyrudin}
W.~Rudin, \emph{Principles of mathematical analysis}, third ed.,
  McGraw-Hill Book Co., New York-Auckland-D\"usseldorf, 1976, International
  Series in Pure and Applied Mathematics.

\bibitem{schleicherzimmer}
D.~Schleicher and J.~Zimmer, \emph{Escaping points of exponential
  maps}, J. London Math. Soc. (2) \textbf{67} (2003), no.~2, 380--400.

\bibitem{urbanskizdunikunstableexp}
M.~Urba{\'n}ski and A.~Zdunik, \emph{Geometry and ergodic theory of
  non-hyperbolic exponential maps}, Trans. Amer. Math. Soc. \textbf{359}
  (2007), no.~8, 3973--3997.

\end{thebibliography}

\providecommand{\bysame}{\leavevmode\hbox to3em{\hrulefill}\thinspace}
\providecommand{\href}[2]{#2}

\end{document}